\documentclass[11pt,oneside,leqno]{amsart}

\makeatletter
\@namedef{subjclassname@2020}{\textup{2020} Mathematics Subject Classification}
\makeatother

\usepackage[top=2.5 cm,bottom=2 cm,left=2 cm,right=3 cm]{geometry}
\usepackage[utf8]{inputenc}
\usepackage{color}
\usepackage{amssymb}

\usepackage[english]{babel}

\usepackage{todonotes}

\usepackage{esint}
\usepackage{graphicx}
\usepackage{hyperref}
\usepackage{bm}
\usepackage{xcolor}

\usepackage{enumerate}

\newcommand*{\mailto}[1]{\href{mailto:#1}{\nolinkurl{#1}}}

\usepackage{todonotes}

\numberwithin{equation}{section}

\newtheorem{example}{Example}[section]

\newtheorem{theorem}[example]{Theorem}
\newtheorem{proposition}[example]{Proposition}

\newtheorem{remark}[example]{Remark}
\newtheorem*{maintheorem*}{Main Theorem}
\allowdisplaybreaks
\numberwithin{equation}{section}

\newcommand\so{
	\mathchoice
	{{\scriptstyle\mathcal{O}}}% \displaystyle
	{{\scriptstyle\mathcal{O}}}% \textstyle
	{{\scriptscriptstyle\mathcal{O}}}% \scriptstyle
	{\scalebox{0.6}{$\scriptscriptstyle\mathcal{O}$}}%\scriptscriptstyle
}

\renewcommand{\i}{\ifmmode\mathit{\mathchar"7010 }\else\char"10 \fi}
\renewcommand{\j}{\ifmmode\mathit{\mathchar"7011 }\else\char"11 \fi}
\newcommand{\R}{\mathbb{R}}
\newcommand{\C}{\mathbb{C}}
\newcommand{\N}{\mathbb{N}}
\newcommand{\Z}{\mathbb{Z}}

\newcommand{\px}{\partial_x}

\newcommand{\ve}{\varepsilon}
\newcommand{\I}{\mathrm{i}}
\newcommand{\Imm}{\mathrm{Im}}
\newcommand{\Ree}{\mathrm{Re}}

{%

	\begin{enumerate}}%
	{\end{enumerate}}

{%

	\begin{enumerate}}%
	{\end{enumerate}}

\begin{document}
	
	\title[NNLS with oscillatory boundary conditions ]{The integrable nonlocal nonlinear Schr\"odinger equation with oscillatory boundary conditions: long-time asymptotics}
	
	%\dedicatory{}	
	
	\author[Rybalko]{Yan Rybalko}
	\author[Shepelsky]{Dmitry Shepelsky}
	\author[Tian]{Shou-Fu Tian}

	\address[Yan Rybalko]{\newline
		Department of Mathematics,  University of Oslo, \newline
		PO Box 1053, Blindern -- 0316 Oslo, Norway
		\medskip
		\newline
		Mathematical Division, 
		B.Verkin Institute for Low Temperature Physics and Engineering
		of the National Academy of Sciences of Ukraine,
		\newline 47 Nauky Ave., Kharkiv, 61103, Ukraine}
	\email[]{rybalkoyan@gmail.com}
	
	\address[Dmitry Shepelsky]{\newline
		Mathematical Division, 
		B.Verkin Institute for Low Temperature Physics and Engineering
		of the National Academy of Sciences of Ukraine,
		\newline 47 Nauky Ave., Kharkiv, 61103, Ukraine}
	\email[]{shepelsky@yahoo.com}
	
	\address[Shou-Fu Tian]{
		\newline School of Mathematics, 
		China University of Mining and Technology,
		\newline
		Xuzhou, 221116, P. R. China
	}
	\email[]{sftian@cumt.edu.cn}

	% 35G25 - Initial value problems for nonlinear higher-order PDEs
	% 35B30 - Dependence of solutions to PDEs on initial and/or boundary data and/or on parameters of PDEs
	% 35Q53 - KdV equations
	% 37K10 - Completely integrable infinite-dimensional Hamiltonian and Lagrangian systems, integration methods, integrability tests, integrable hierarchies (KdV, KP, Toda, etc.)
	
	
	%\thanks{DS acknowledges the  support from the {QJMAM} Fund for Applied Mathematics.
		%	The work of YR was supported by the Marie Sk\l{}odowska-Curie Actions Postdoctoral Fellowship, Grant agreement ID: 101058830.
		%	YR and DS  thank the Armed Forces of Ukraine for providing security, which made this work possible.}
	
	\date{\today}
	
	\begin{abstract}
		We consider the Cauchy problem for 
		the integrable nonlocal nonlinear Schr\"odinger equation
		\[
		\I q_{t}(x,t)+q_{xx}(x,t)+2 q^{2}(x,t)\bar{q}(-x,t)=0,
		\]
		subject to the step-like initial data: $q(x,0)\to0$ as $x\to-\infty$ and 
		$q(x,0)\simeq Ae^{2\I Bx}$ as $x\to\infty$, where $A>0$ and $B\in\mathbb{R}$.
		The  goal is to study
		the long-time asymptotic behavior of the solution of this  problem
		assuming that $q(x,0)$ is close, in a certain spectral sense, to the ``step-like'' 
		function $q_{0,R}(x)=
		\begin{cases}
			0, &x\leq R,\\
			Ae^{2\I Bx}, &x>R,
		\end{cases}$
		with $R>0$.
		A special attention is paid to how $B\ne0$ affects the asymptotics.
	\end{abstract}
	
	\maketitle
	
	\tableofcontents
	
	\section{Introduction}
	
	Nonlinear evolution partial differential equations with nonzero boundary conditions as the spatial variable approaches infinity, particularly when the conditions differ at opposite infinities (known as step-like conditions), have captured the attention of researchers, physicists as well as mathematicians, since the 1970s \cite{Bik89,GP73,GP74,Khr75,Khr76,Ven86}. This interest stems from the remarkable properties of their solutions, including the wide variety of patterns observed in their large-time behavior. In particular, when the nonlinear equation turns to be  integrable (in the sense that it possesses the Lax pair representation), it is the inverse scattering transform (IST) method that have proved its extreme efficiency in studying the relevant asymptotic pictures in details.
	
	It is generally recognized that the simplest (although fundamental in many aspects) scalar integrable nonlinear equations arise as reductions of the two-component system of nonlinear Schr\"odinger equations (a.k.a.\,\,the Ablowitz--Kaup--Newell--Segur (AKNS) system \cite{AKNS})
	\begin{equation}
		\label{ss}
		\begin{split}
			& \I q_t(x,t)+q_{xx}(x,t)-2q^2(x,t)r(x,t)=0,\\
			&  \I r_t(x,t)-r_{xx}(x,t)+2r^2(x,t)q(x,t)=0.
		\end{split}
	\end{equation}
	Namely, the reductions $r(x,t)=\bar q(x,t)$ and $r(x,t)=-\bar q(x,t)$ of \eqref{ss}
	(bar stands for the complex conjugation)
	give rise to respectively the defocusing and focusing nonlinear Schr\"odinger 
	(NLS) equations
	\begin{equation}
		\label{NLS}
		\I q_{t}(x,t)+q_{xx}(x,t)\mp 2|q(x,t)|^2q(x,t)=0.
	\end{equation}
	The IST method for the NLS equations (on the zero background) was developed by
	Vladimir Zakharov and Alexei Shabat in 1971 \cite{ZSh71}.
	
	On the other hand, the reductions $r(x,t)=\pm \bar q(-x,t)$
	lead to the nonlocal (two-place) versions of these equations
	\begin{equation}
		\label{NNLS-2}
		\I q_{t}(x,t)+q_{xx}(x,t)\mp 2 q^2(x,t) \bar q(-x,t)=0,
	\end{equation} 
	whose detailed study by the IST method was initiated  by Ablowitz and Musslimani in \cite{AM13}.
	We provide further contextual information about the NNLS equations later in the introduction.
	
	Since the conventional (local) nonlinear Schr\"odinger equations \eqref{NLS} don't support constant solutions,
	the simplest nonzero solutions for them have the form $q(x,t)=Ae^{2\I Bx+2\I\omega t+\I \phi}$,
	where the (real) constants $A\ge0 $, $B$, $\omega$, and $\phi$ are related as $\omega=\mp A^2-2B^2$.
	Respectively, the simplest non-decaying initial conditions $q_0(x)$ for the Cauchy problem
	for the nonlinear Schr\"odinger equations are those approaching the corresponding exponentials
	fast enough when $x\to\pm\infty$:
	\begin{equation*}
		q_0(x)\simeq
		\begin{cases}
			A_1e^{2\I B_1x + \I \phi_1}, & x\to -\infty,\\
			A_2e^{2\I B_2x +\I \phi_2}, & x\to \infty.\\
		\end{cases}
	\end{equation*}
	
	Regarding the local nonlinear Schr\"odinger equations \eqref{NLS},
	their focusing and defocusing versions are substantially different in many aspects. Particularly, this is due to the fact that the Zakharov-Shabat systems playing the role of the spatial equation of the Lax pair are either
	formally self-adjoint (in the defocusing case) or not (in the focusing case).
	In the defocusing case, the IST method for the problems with the boundary conditions of the finite density type ($A_1=A_2\ne0$, $b_1=B_2=0$) was initiated 
	by Zakharov and Shabat in \cite{ZSh73}. Since then, an extensive literature devoted to problems  on nonzero backgrounds has been developed. Without pretending to give a comprehensive review, we note that the solution of the defocusing NLS equation with asymmetric ($A_1\ne A_2$) nonzero boundary conditions was studied by IST methods in \cite{BP1982}
	(see also \cite{BFP16}), and extensive results on its long-term behavior were presented in \cite{FLQ24, J2015}.
	
	The first studies of the focusing NLS equation with nonzero boundary conditions by the IST method were presented in \cite{KI78,Ma79}, with  a single background that can be characterized by $A_1=A_2$, $\phi_1=\phi_2$, and $B_1=B_2=0$ (i.e., the solution is assumed to approach $Ae^{\mp 2\I A^2 t}$). In particular, the Ma soliton \cite{Ma79} (also discovered in \cite{KI78}) was introduced. It was also mentioned in \cite{Ma79} that a plane wave solution corresponds to a one-band potential in the spectrum of the Zakharov--Shabat scattering equations, whereas the cnoidal wave (elliptic function) and the multicnoidal wave (hyperelliptic function) solutions correspond to two-band and $N$-band potentials, respectively. A perturbation theory for the NLS equation with non-vanishing boundary conditions was put forward in \cite{GK12}, where particular attention was paid to the stability of the Ma soliton. The results of the Whitham theory for the focusing NLS with step-like data can be found in \cite{B1995}.
	
	An IST approach for initial data  with $A_1=A_2$, $\phi_1\ne\phi_2$, and $B_1=B_2=0$ was presented in \cite{BK14}, and was further developed in \cite{BM16,BM17}. In particular, it was shown in \cite{BM16,BM17} that for such initial data, the long-time behavior is described by three asymptotic sectors in the $(x,t)$ half-plane $t>0$: two sectors adjacent to the half-axes $x<0$, $t=0$ and $x>0$, $t=0$ in which the solution asymptotes to modulated plane waves, and a middle sector in which the solution asymptotes to an elliptic (genus~$1$) modulated wave. The IST formalism for the case of {asymmetric} nonzero boundary conditions ($A_1\neq A_2$,  $B_1=B_2=0$) was presented in \cite{D14}. 
	
	In \cite{BLS21} (see also \cite{BV07}), the long-time asymptotics for the focusing NLS equation was studied for the so-called ``shock case'' characterized by $B_1=-B_2 > 0$ (which can be achieved, by a linear change of independent variables starting with any $B_1>B_2$), with  $A_1=A_2$ and $\phi_1=\phi_2$
	(notice that $B$ in the present paper corresponds to $-B$ in \cite{BLS21}). In this case, 
	a rich panorama of asymptotic scenarios  was discovered.
	Namely,  whereas the long-time
	behavior along the $t$-axis is always described by a genus~1 wave, the asymptotics along
	the lines $x/t = const$ can be either a genus~0, a genus~1, a genus~2,
	or a genus~3 wave depending on the value of $A_j/(B_2 - B_1)$.

	The long-time asymptotics in the case when the left background is zero (i.e., when $A_1=0$ and $A_2\neq 0$) was analyzed in \cite{BKS11}. It was shown that the asymptotic picture involves three sectors in this case: a slow decay sector (adjacent to the negative $x$-axis), a modulated plane wave sector (adjacent to the positive $x$-axis), and a modulated elliptic wave sector (between the first two).

	Returning to the nonlocal NLS equations \eqref{NNLS-2}, we notice that it satisfies the parity-time (PT) symmetry condition: $q(x,t)$ is a solution of the NNLS equations along with $\bar{q}(-x,-t)$ \cite{AM13}. 
	Therefore the NNLS equations are closely related to the theory of PT-symmetric and non-Hermitian systems, which describe various phenomena in both classical and quantum physics, see, e.g., \cite{BKM, B, EMKM18, GA16, KYZ, ZB} and references therein.
	The other examples of the nonlocal reductions of the two-component integrable systems can be found in, e.g., \cite{AM17, LQ17}.
	
	The study of problems with step-like initial data for nonlocal equations was initiated in \cite{RS21-DE}, in the case of the initial data satisfying the boundary conditions
	\begin{equation*}
		q_0(x)\simeq
		\begin{cases}
			0, & x\to -\infty,\\
			A & x\to \infty,\\
		\end{cases}
	\end{equation*}
	and assuming that this large-$x$ behavior is preserved for any $t$.
	Notice that in contrast with the local NLS equations, the behavior of a solution to the nonlocal NLS equation at one infinity cannot be prescribed independently of its behavior at another infinity. In particular,
	assuming the solution to satisfy
	\begin{equation}\label{bc}
		q(x,t)=\so(1),
		\quad x\to-\infty,\quad
		q(x,t)=Ae^{2\I Bx+2\I\omega t}+
		\so(1),
		\quad x\to\infty,
	\end{equation}
	for all $t\in\mathbb{R}$
	and substituting the ``limiting'' functions into the ``focusing'' NNLS equation (see \eqref{NNLS} below),
	the parameters in the second condition in \eqref{bc} have to satisfy
	$2B^2+\omega = 0$ and thus \eqref{bc} reduces to 
	\begin{equation}\label{bc2}
		q(x,t)=\so(1),
		\quad x\to-\infty,\quad
		q(x,t)=Ae^{2\I Bx-4\I B^2 t}+
		\so(1),
		\quad x\to\infty.
	\end{equation}

	In the present paper we consider the following Cauchy problem for the  nonlocal nonlinear Schr\"odinger  equation
	%(here and below $\bar{q}$ is the complex conjugate of $q$)

	\begin{subequations}
		\label{IVP}
		\begin{align} \label{NNLS}
			& \I q_{t}(x,t)+q_{xx}(x,t)
			+2q^{2}(x,t)\bar{q}(-x,t)=0,
			\quad x,t\in\mathbb{R},\quad \I^2=-1, \\
			\label{IV}
			& q(x,0)=q_0(x), & &  x\in\mathbb{R}
		\end{align}
	\end{subequations}
	in the class of functions $q(x,t)$ satisfying the  boundary conditions \eqref{bc2}.
	Accordingly, the initial data $q_0(x)$ is assumed to be a (generalized) 
	step-like function:
	\begin{equation}\label{ic}
		q_0(x)=\so(1),\quad x\to-\infty,
		\quad
		q_0(x)=Ae^{2\I Bx}+\so(1),\quad
		x\to\infty.
	\end{equation}
	Notice that, in contrast with the local NLS equations, the ``focusing'' and ``defocusing'' variants of the NNLS equations \eqref{NNLS-2} with boundary conditions \eqref{bc2} with $B=0$
	exhibit quite similar properties (cf. \cite{RS20-JMAG} and \cite{RS21-CMP}).
	
	In \cite{RS21-DE, RS21-SIAM, RS21-CMP} the step-like problems for the NNLS equation were 
	considered in the case $B=0$ (concerning the NNLS with nonzero background see, e.g., \cite{ALM18, WT22, XXLH25, XLLLZ21} and references therein).
	Particularly, in \cite{RS21-CMP}
	it was shown that there is a variety of scenarios of the long-time behavior
	of solutions of step-like problems depending not only on the parameters
	of the step (as it takes place in the case of the local NLS equation), but also on details of the initial data, such that the ``location'' of the 
	step-like form of the initial function. The latter is because the NNLS equation 
	is not translation invariant, which is in contrast with their local counterparts.
	Consequently, considering the initial data close, in some sense (the exact meaning will be given in what follows), to the ``shifted step'' initial data
	\begin{equation}\label{shst}
		q_{0,R}(x)=
		\begin{cases}
			0 &x\leq R,\\
			Ae^{2\I Bx} &x>R,
		\end{cases}
	\end{equation}
	with some $R\geq0$, it is natural to expect that the resulting large-time behavior 
	of $q(x,t)$ will depend on $R$. 
	
	The main goal of this paper is to investigate the effect of non-zero $B$ on the asymptotic picture for $q(x,t)$ depending on $A$ and $R$. 
	The following proposition illustrates this effect in the case of pure step initial values:
	
	\begin{proposition}[Rough asymptotics for pure step initial values]
		\label{RA1}
		Consider the Cauchy problem \eqref{IVP} with boundary conditions \eqref{bc2}.
		Assume that the initial datum $q_0(x)$ has the form \eqref{shst} with $0<R<\frac{\pi}
		{2\left(4B^2+A^2\right)^{1/2}}$, 
		$0< 4|B|R\leq\pi$
		and that either (i) $4B^2-A^2<0$ or (ii) $4B^2-A^2>0$ and 
		$\pi^{-2}R^2(4B^2-A^2)$ is not a squared integer.
		Then we have the following rough (up  to decaying terms $\so(1)$) long-time asymptotic behavior of $q(x,t)$, as 
		$t\to\infty$, along the rays $\xi=\frac{x}{4t}=const$, depending on the sign of $B$ (see Figure \ref{fas}):
		\begin{enumerate}[1)]
			\item If $B>0$, then
			\begin{equation*}
				q(x,t)=
				\begin{cases}
					\so(1),&\xi<B,\,\xi\neq -B,\,\xi\neq0\\
					A\delta^2(-B,\xi)
					e^{2\I Bx-4\I B^2t}
					+\so(1),&\xi>B;
				\end{cases}
			\end{equation*}
			
			\item If $B<0$, then
			\begin{equation*}
				q(x,t)=
				\begin{cases}
					\so(1),&\xi<B,\\
					\frac{16AB^2\hat{\delta}^2(-B,\xi;B)
						e^{2\I Bx-4\I B^2t}}
					{16B^2
						-A^2\left(\overline{\hat{\delta}}\right)^2
						(-B,-\xi;B)
						\hat{\delta}^2(-B,\xi;B)
						e^{4\I Bx}}
					+\so(1),&0<|\xi|<-B,\\
					A\delta^2(-B,\xi)
					e^{2\I Bx-4\I B^2t}
					+\so(1),&\xi>-B;
				\end{cases}
			\end{equation*}
		\end{enumerate}
		Here $\delta(-B,\xi)$ and 
		$\hat{\delta}(-B,\xi;B)$ are given in \eqref{ddef} and \eqref{hat-del}, respectively, in terms of the associated spectral data.
		The asymptotic formula in item 2 for $0<|\xi|<-B$ holds uniformly in $x,t$ away from  arbitrarily small neighborhoods of possible zeros of the denominator.
		
		Moreover, in the case (ii), i.e., if $4B^2-A^2>0$ and 
		$\pi^{-2}R^2(4B^2-A^2)$ is not a squared integer, we  obtain the asymptotics along the ray $\xi=0$ for both $B>0$ and $B<0$, which is consistent with those in the sectors $0<|\xi|<|B|$ (see  Theorem \ref{ThCII0} below for details).
	\end{proposition}
	\begin{remark}
		Notice that the asymptotics in the middle sector in Proposition \ref{RA1}, item 2),
		is described in terms of  periodic elementary functions.
		This  contrasts with the conventional NLS equations, where the asymptotics in sectors between the plane wave sectors exhibits more complicated behavior involving theta functions of various genera \cite{BLS21, BV07, J2015}.
	\end{remark}
	\begin{figure}[h]
		\begin{minipage}[h]{0.48\linewidth}
			\centering{\includegraphics[width=0.99\linewidth]{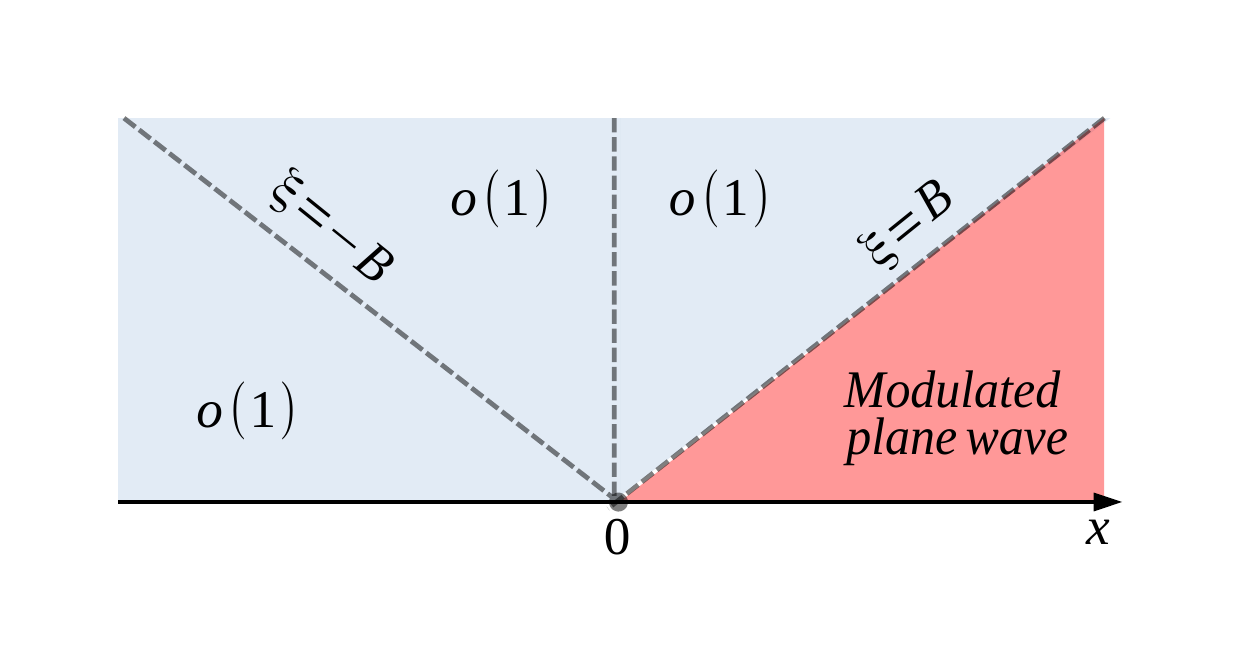}}
		\end{minipage}
		\hfill
		\begin{minipage}[h]{0.48\linewidth}
			\centering{\includegraphics[width=0.99\linewidth]{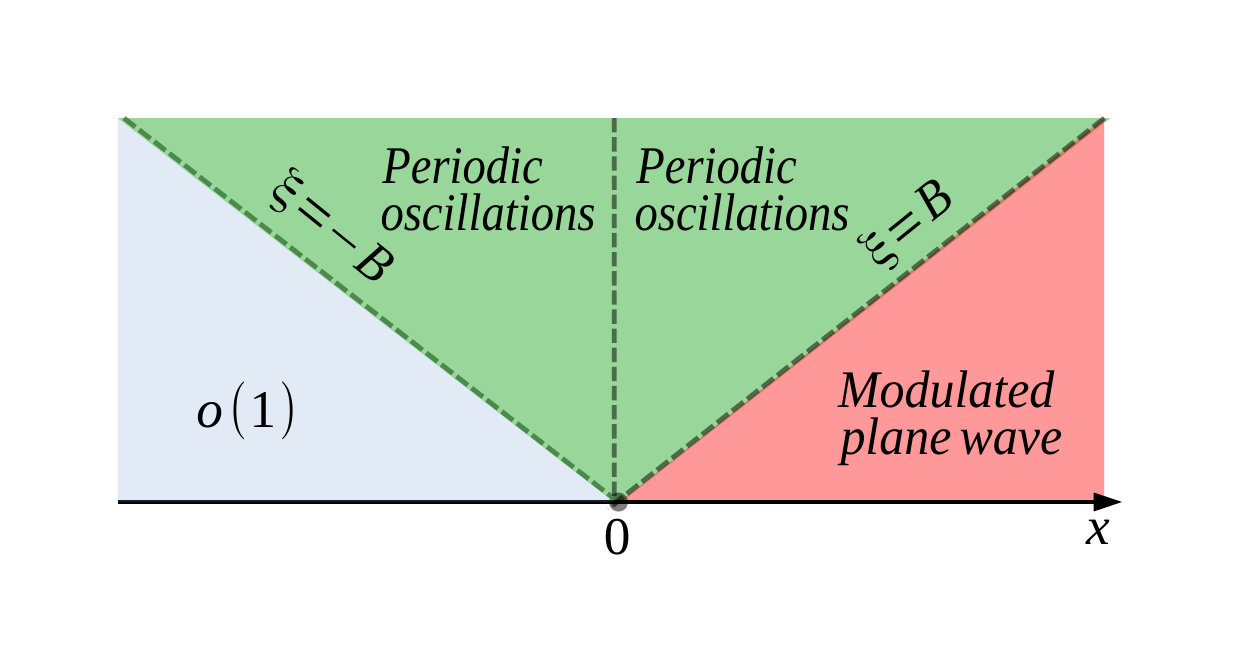}}
		\end{minipage}
		\caption{Asymptotics of the pure step initial data, described in Proposition \ref{RA1}.
			The left figure illustrates the case $B>0$, and the right corresponds to $B<0$.}
		\label{fas}
	\end{figure}
	
	The paper is organized as follows. Our analysis is based on a RH formalism of the IST method, which is
	developed in Section 2. A detailed account of the case of   pure 
	step initial data \eqref{shst} is given in Subsection 2.3, where 
	special attention is payed to the location of zeros of the dedicated  
	spectral function and the associated winding of its argument. These properties 
	are then used in the formulation of assumptions which characterize the ``closeness''
	of  general initial data to a particular pure step one in the sense that 
	the associated spectral data share similar properties.
	%the 
	%long-time asymptotics of the solutions of the corresponding
	%Cauchy problems shares the same qualitative behavior.
	In Section 3, we analyze the long-time behavior of the solution of
	the Cauchy problem \eqref{IVP} with boundary conditions \eqref{bc2}.
	Particularly, in Subsections \ref{S1}--\ref{S2}, we present a detailed proof of the 
	asymptotics of the solution with the initial data close to the pure step considered in Proposition \ref{RA1} (i.e., with relatively small $R>0$).
	Referring to the rigorous asymptotic analysis presented in \cite{RS21-DE, RS21-SIAM}, we specify the decaying terms in the asymptotics and obtain Theorems \ref{ThCII0}, \ref{ThCI0}.
	%Then in Subsection 3.3,
	%we briefly discuss another case with relatively small $R$.
	In Concluding Remarks (Section \ref{CR}) we present the long-time asymptotic formulas
	in the case of the initial data close to pure step \eqref{shst} with a general value of $R>0$, see Proposition \ref{pr1}.
	Finally, Appendices \ref{Ap1} and \ref{Ap2} present
	details on deriving properties of the spectral function $a_1$ stated in Propositions \ref{a1Zs} and \ref{a1Ws}.
	
	\bigskip
	
	\noindent\textbf{Notations.} $I$ is the $2\times 2$ identity matrix;
	$\mathbf{0}_{2\times2}$ denotes the zero $2\times2$ matrix;
	$\sigma_1=\left(\begin{smallmatrix} 0& 1\\ 1 & 0\end{smallmatrix}\right)$ and 
	$\sigma_3=\left(\begin{smallmatrix} 1& 0\\ 0 & -1\end{smallmatrix}\right)$ are the Pauli matrices;
	$\C^\pm = \{k: \pm\Imm\,k>0\}$;
	$\overline{\C^\pm}=\C^\pm\cup\R$; 
	the set of numbers 
	$\{z_i\}_{i=n_0}^{n_1}$, $n_0,n_1\in\Z$ is empty if 
	$n_1<n_0$ and $\prod\limits_{s=m_1}^{m_2}F_s=1$ if $m_1>m_2$;
	$X^{(i)}$ is the $i$-th column of a matrix $X$;
	$X_{i,j}$ is the $i,j$-th entry of a matrix $X$.
	For any $a\in\R$ we write 
	$k\to a\pm\I0$ if 
	$k\to a$ and $k\in\C^\pm$.
	
	\section{Inverse scattering transform and the Riemann-Hilbert problem}
	\subsection{Eigenfunctions}
	Being a particular case of the AKNS system,
	the NNLS equation \eqref{NNLS} has a Lax pair, which reads
	%is integrable equation, i.e. it is a compatibility condition of the following two linear equations (Lax pair)
	\begin{equation}
		\label{LP}
		\begin{split}
			&\Phi_{x}+\I k\sigma_{3}\Phi=U(x,t)\Phi,\\
			&\Phi_{t}+2\I k^{2}\sigma_{3}\Phi=V(x,t,k)\Phi,
		\end{split}
	\end{equation}
	where $\Phi(x,t,k)$ is a $2\times2$ matrix-valued function, $k\in\mathbb{C}$ is an auxiliary (spectral) parameter and the matrix coefficients $U(x,t)$ and $V(x,t,k)$ are given in terms of a function $q(x,t)$:	
	\begin{equation}
		\label{U}
		U(x,t)=\begin{pmatrix}
			0& q(x,t)\\
			-\bar{q}(-x,t)& 0\\
		\end{pmatrix},
	\end{equation}
	\begin{equation}
		\label{V}
		V(x,t,k)=\begin{pmatrix}
			\tilde{A}& \tilde{B}\\
			\tilde{C}& -\tilde{A}\\
		\end{pmatrix},
	\end{equation}
	with $\tilde{A}=\I q(x,t)\bar{q}(-x,t)$, $\tilde{B}=2kq(x,t)+\I q_{x}(x,t)$, $\tilde{C}=-2k\bar{q}(-x,t)+\I (\bar{q}(-x,t))_{x}$.
	Direct calculations show that the compatibility condition of \eqref{LP}
	(here $[X,Y]=XY-YX$ is the matrix commutator)
	\begin{equation}\label{cc}
		U_t-V_x
		+[U-\I k\sigma_3,V-2\I k^2\sigma_3]=0,
	\end{equation}
	is equivalent to \eqref{NNLS}.
	
	Introduce $U_\pm$ and $V_\pm$ as the limits of $U$ and $V$, respectively, as $x\to\pm\infty$ with the fixed $t,k$:
	\begin{equation*}
		%\lim\limits_{x\to\pm\infty}
		%\left(U-U_{\pm}\right)(x,t)=0,\quad
		%\lim\limits_{x\to\pm\infty}
		%\left(V-V_{\pm}\right)(x,t,k)=0.
		U(x,t)=U_{\pm}(x,t)+\so(1),\,\,
		x\to\pm\infty,\quad
		V(x,t,k)=V_{\pm}(x,t,k)+\so(1),\,\,
		x\to\pm\infty,\quad
	\end{equation*}
	Taking into account boundary conditions \eqref{bc2},
	we conclude from \eqref{U}, \eqref{V} that
	$U_\pm$ and $V_\pm$ have the following form:
	\begin{align}
		\label{Upm}
		&U_+(x,t)=
		\begin{pmatrix}
			0 & Ae^{2\I Bx-4\I B^2 t}\\
			0 & 0
		\end{pmatrix},
		&&U_-(x,t)=
		\begin{pmatrix}
			0 & 0\\
			-Ae^{2\I Bx+4\I B^2 t} & 0
		\end{pmatrix},\\
		\nonumber
		&V_+(x,t,k)=
		\begin{pmatrix}
			0 & 2A(k-B)e^{2\I Bx-4\I B^2 t}\\
			0 & 0
		\end{pmatrix},
		&&V_-(x,t,k)=
		\begin{pmatrix}
			0 & 0\\
			-2A(k+B)e^{2\I Bx+ 4\I B^2 t} & 0
		\end{pmatrix}.
	\end{align}

	%Of course, the constant $A$ is not an exact solution of \eqref{NNLS}, but the nonlocal term 
	%$q^2(x,t)\bar{q}(-x,t)$
	%``mixes up'' the values of the solution at different infinities and it turns out that
	%step-like boundary condition \eqref{sA} is a
	
	%Such a boundary condition is admissible, though the constant $A$ is not an exact solution of the NNLS equation, since \eqref{NNLS} involves values of $q$ at both $x$ and $-x$, which ``mixes up'' the behavior of the solution at different infinities.
	%The long-time asymptotics of such a problem was studied in \eqref{}.
	%$q(x,t)\to 0$, as $x\to-\infty$ and $q(x,t)\to A$, as $x\to\infty$,
	%is admissible, even though the constant $A$ is not an exact solution of \eqref{NNLS}.
	%The long-time asymptotic behavior of such a problem
	%for the NNLS equation was studied in \eqref{}.
	
	%Notice that in the case of the conventional NLS equation with step-like boundary 
	%conditions \eqref{bc}, the corresponding relation for $A,B,\omega$ has the form \cite{BKS11}
	%\begin{equation*}
	%	2B^2+2\omega-A^2=0.
	%\end{equation*}
	%Here the amplitude $A>0$ depends on the frequencies $B$ and $\omega$.
	
	Introduce the following functions:
	\begin{equation}
		\label{Phipm}
		\Phi_{\pm}(x,t,k)=
		e^{(\pm \I Bx-2\I B^2 t)\sigma_3}N_{\pm}(k)
		e^{-(\I(k\pm B)x+2\I(k^2-B^2)t)\sigma_3},
	\end{equation}
	with
	\begin{equation}
		\label{Npm}
		N_+(k)=
		\begin{pmatrix}
			1 & \frac{-\I A}{2(k+B)}\\
			0 & 1
		\end{pmatrix},
		\quad
		N_-(k)=
		\begin{pmatrix}
			1 & 0\\
			\frac{-\I A}{2(k-B)} & 1
		\end{pmatrix}.
	\end{equation}
	Then direct calculations show that $\Phi_\pm$
	solves \eqref{LP} with  $U$ and $V$ replaced by  $U_{\pm}$ and $V_{\pm}$  respectively.
	
	Define (formally) the $2\times2$ matrix-valued functions $\Psi_j(x,t,k)$, $j=1,2$, as the solutions of the following Volterra integral equations:
	\begin{equation}
		\label{Psi1}
		\begin{split}
			\Psi_1(x,t,k)=&\,
			e^{(-\I Bx-2\I{B^2} t)\sigma_3}N_-(k)\\
			&+
			\int_{-\infty}^{x}G_-(x,y,t,k)
			\left(U-U_-\right)(y,t)
			\Psi_1(y,t,k)e^{\I(k-B)(x-y)\sigma_3}\,dy,
		\end{split}
	\end{equation}
	\begin{equation}
		\label{Psi2}
		\begin{split}
			\Psi_2(x,t,k)=&\,
			e^{(\I Bx-2\I{B^2} t)\sigma_3}N_+(k)\\
			&-
			\int^{\infty}_{x}G_+(x,y,t,k)
			\left(U-U_+\right)(y,t)
			\Psi_2(y,t,k)e^{\I(k+B)(x-y)\sigma_3}\,dy,
		\end{split}
	\end{equation}
	where $N_\pm$ and $U_\pm$ are given in \eqref{Npm} and \eqref{Upm}, respectively, and
	(see \eqref{Phipm})
	\begin{equation*}
		%\label{Gpm}
		G_\pm(x,y,t,k)=
		\Phi_\pm(x,t,k)
		\Phi_\pm^{-1}(y,t,k).
	\end{equation*}
	The functions $\Psi_j(x,t,k)$, $j=1,2$, will be the main ingredients of the basic RH problem (see below).
	
	Let us establish the connection between $\Psi_j$, $j=1,2$, and the solutions of Lax pair \eqref{LP}.
	Also we obtain important analytical and symmetry properties of matrices $\Psi_j$, $j=1,2$.
	Here and below we will denote by $X^{(i)}$ the $i$-th column of matrix $X$, and $\mathbb{C}^{\pm}=\left\{k\in\mathbb{C}\,|\pm
	\mathrm{Im} k>0\right\}$.
	\begin{proposition}
		\label{PPsi}
		Assume that $xq(x,t)\in L^1(-\infty,a)$ and
		$\left(
		q(x,t)-Ae^{2\I B(\cdot)-4\I{B^2} t}
		\right)\in L^1(a,\infty)$ with respect to the spatial variable $x$, for all fixed $t\in\mathbb{R}$ and $a\in\mathbb{R}$.
		Then the matrices $\Psi_j(x,t,k)$, $j=1,2$, given in \eqref{Psi1} and \eqref{Psi2}, have the following properties:
		\begin{enumerate}
			
			\item the functions $\Phi_j(x,t,k)$, $j=1,2$, defined by
			\begin{equation}
				\label{PhPs}
				\Phi_j(x,t,k)=
				\Psi_j(x,t,k)e^{-(\I (k+(-1)^jB)x+2\I(k^2{-B^2})t)\sigma_3},
				\quad k\in\mathbb{R}\setminus \{(-1)^{j+1}B\},\, j=1,2,
			\end{equation}
			are the (Jost) solutions of Lax pair \eqref{LP}, which satisfy the following boundary conditions for all fixed $t,k\in\mathbb{R}$ (recall \eqref{Phipm}):
			\begin{equation*}
				\begin{split}
					&\Phi_1(x,t,k)=\Phi_-(x,t,k)+\so(1),
					\quad x\rightarrow-\infty,\\
					&\Phi_2(x,t,k)=\Phi_+(x,t,k)+\so(1),
					\quad x\rightarrow\infty;
				\end{split}
			\end{equation*}
			
			\item the columns $\Psi_1^{(1)}(x,t,k)$ and $\Psi_2^{(2)}(x,t,k)$ are analytic in $\mathbb{C}^+$ and 
			\begin{equation}\label{Psas2}
				\begin{split}
					&\Psi_1^{(1)}(x,t,k)=
					\begin{pmatrix}
						1\\
						0\end{pmatrix}
					+\mathcal{O}\left(k^{-1}\right),\quad k\in\mathbb{C}^+,\,
					k\rightarrow\infty,\\
					&\Psi_2^{(2)}(x,t,k)=
					\begin{pmatrix}
						0\\
						1\end{pmatrix}
					+\mathcal{O}\left(k^{-1}\right),\quad k\in\mathbb{C}^+,\,
					k\rightarrow\infty;
				\end{split}
			\end{equation}
			
			\item the columns $\Psi_1^{(2)}(x,t,k)$ and $\Psi_2^{(1)}(x,t,k)$ are analytic in $\mathbb{C}^-$ and 
			\begin{equation}\label{Psas1}
				\begin{split}
					&\Psi_1^{(2)}(x,t,k)=
					\begin{pmatrix}
						0\\
						1\end{pmatrix}
					+\mathcal{O}\left(k^{-1}\right),\quad k\in\mathbb{C}^-,\,
					k\rightarrow\infty,\\
					&\Psi_2^{(1)}(x,t,k)=
					\begin{pmatrix}
						1\\
						0\end{pmatrix}
					+\mathcal{O}\left(k^{-1}\right),\quad
					k\in\mathbb{C}^-,\,
					k\rightarrow\infty;
				\end{split}
			\end{equation}
			
			\item $\Psi_j(x,t,k)$, $j=1,2$, satisfy the following symmetry relation:
			\begin{equation}
				\label{Psymm}
				\sigma_1\overline{\Psi_1}(-x,t,-k)\sigma_1^{-1}
				=\Psi_2(x,t,k),\,\,k\in\mathbb{R}\setminus\{B,-B\},
			\end{equation}
			where $\sigma_1=\bigl(\begin{smallmatrix}0& 1\\1 & 0\end{smallmatrix}\bigl)$ is the first Pauli matrix;
			
			\item $\Psi_1(x,t,k)$ and $\Psi_2(x,t,k)$
			have the following asymptotic expansion at the points $k=B$ and $k=-B$, respectively:
			\begin{equation}\label{Psi12B}
				\begin{split}
					&\Psi_1^{(1)}(x,t,k)=
					\begin{pmatrix}\frac{v_1(x,t)}{k-B}\\\frac{v_2(x,t)}{k-B}
					\end{pmatrix}+\mathcal{O}(1),\,
					\Psi_1^{(2)}(x,t,k)=
					\frac{2\I}{A}
					\begin{pmatrix}v_1(x,t)\\ v_2(x,t)\end{pmatrix}
					+\mathcal{O}(k-B),\,k\to B,
					\\
					&\Psi_2^{(1)}(x,t,k)=-\frac{2\I}{A}
					\begin{pmatrix}
						\overline{v_2}(-x,t)\\ \overline{v_1}(-x,t)
					\end{pmatrix}+\mathcal{O}(k+B),\,
					\Psi_2^{(2)}(x,t,k)=
					\begin{pmatrix}
						\frac{-\overline{v_2}(-x,t)}{k+B}\\
						\frac{-\overline{v_1}(-x,t)}{k+B}
					\end{pmatrix}+\mathcal{O}(1),\,k\to -B,
				\end{split}
			\end{equation}
			where $v_j(x,t)$, j=1,2, solve the following linear Volterra integral equations:
			\begin{equation}\label{v1v2}
				\begin{split}
					&v_1(x,t)=\int_{-\infty}^{x}
					e^{-\I B(x-y)}q(y,t)v_2(y,t)\,dy,\\
					&v_2(x,t)=
					-\I\frac{A}{2}e^{\I Bx+2\I{B^2} t}
					+e^{\I Bx}\int_{-\infty}^{x}
					\left(
					Ae^{\I By}(e^{4\I{B^2} t}-1)
					-e^{-\I By}\bar{q}(-y,t)
					\right)v_1(y,t)\,dy;
				\end{split}
			\end{equation}
			
			\item $\det\Psi_j(x,t,k)=1$,$\quad$ $x,t\in\mathbb{R}$, 
			$k\in\mathbb{R}\setminus\{B,-B\}$, 
			$j=1,2$.
		\end{enumerate}
	\end{proposition}
	\begin{proof}
		Items (1)--(3) follow directly from the definitions of $\Psi_j$, $j=1,2$, given in \eqref{Psi1}, \eqref{Psi2}.
		
		Item (4) follows from the symmetry relations
		$$
		\sigma_1 e^{a\sigma_3}\sigma_1^{-1}
		=e^{-a\sigma_3},\quad
		\sigma_1 \overline{N_-}(-k)\sigma_1^{-1}
		=N_+(k),\quad
		\sigma_1 \overline{U}(-x,t)\sigma_1^{-1}
		=-U(x,t).
		$$
		
		Let us prove item (5).
		Assume that $\Psi_1(x,t,k)$ has the following asymptotic behavior as $k\to B$:
		\begin{equation}\label{Psi1b}
			\Psi_1^{(1)}(x,t,k)=\begin{pmatrix}
				\frac{v_1(x,t)}{k-B}\\ \frac{v_2(x,t)}{k-B}
			\end{pmatrix}+\mathcal{O}(1),\quad
			\Psi_1^{(2)}(x,t,k)=\begin{pmatrix}w_1(x,t)\\ w_2(x,t)\end{pmatrix}
			+\mathcal{O}(k-B),\,k\to B.\\
		\end{equation}
		From symmetry \eqref{Psymm} we conclude that the expansion of $\Psi_2(x,t,k)$ as $k\to-B$ is given by
		\begin{equation}\label{Psi2b}
			\Psi_2^{(1)}(x,t,k)=
			\begin{pmatrix}
				\overline{w_2}(-x,t)\\ \overline{w_1}(-x,t)
			\end{pmatrix}+\mathcal{O}(k+B),\,
			\Psi_2^{(2)}(x,t,k)=
			\begin{pmatrix}
				\frac{-\overline{v_2}(-x,t)}{k+B}\\
				\frac{-\overline{v_1}(-x,t)}{k+B}
			\end{pmatrix}+\mathcal{O}(1),\,k\to -B.
		\end{equation}
		
		Observe that $G_-(x,y,t,k)$ can be written in the following form:
		\begin{equation*}
			G_-(x,y,t,k)=
			\begin{pmatrix}
				e^{-\I B(x-y)}&0\\
				-A(x-y)e^{\I B(x+y)}&e^{\I B(x-y)}
			\end{pmatrix}
			+\mathcal{O}(k-B),\quad
			k\to B.
		\end{equation*}
		Comparing the coefficients of $(k-B)^{-1}$ for the first column in \eqref{Psi1}, we conclude that
		$v_j(x,t)$, $j=1,2$,
		satisfy the following system of integral equations:
		\begin{equation}\label{v1v2b}
			\begin{split}
				&v_1(x,t)=\int_{-\infty}^x
				e^{-\I B(x-y)}q(y,t)v_2(y,t)\,dy,\\
				&v_2(x,t)e^{-\I Bx}=
				-\I\frac{A}{2}e^{2\I{B^2} t}
				+\int_{-\infty}^x
				\left(
				Ae^{\I By+4\I{B^2} t}
				-e^{-\I By}\bar{q}(-y,t)
				\right)
				v_1(y,t)\,dy
				\\
				&\qquad\qquad\qquad\,\,\,
				-A\int_{-\infty}^x(x-y)
				e^{\I By}q(y,t)v_2(y,t)\,dy.
			\end{split}
		\end{equation}
		Taking the derivative in $x$ of the second equation in \eqref{v1v2b}, we obtain
		\begin{equation*}
			\px\left(v_2(x,t)e^{-\I Bx}\right)
			=\left(
			Ae^{\I Bx}
			\left(e^{4\I{B^2} t}-1\right)
			-e^{-\I Bx}\bar{q}(-x,t)
			\right)
			v_1(x,t),
		\end{equation*}
		which, together with \eqref{v1v2b}, implies \eqref{v1v2}.
		
		Then equating the coefficients of $(k-B)^0$ for the second column in \eqref{Psi1}, we obtain
		\begin{equation}\label{w1w2}
			\begin{split}
				&w_1(x,t)=\int_{-\infty}^x
				e^{-\I B(x-y)}q(y,t)w_2(y,t)\,dy,\\
				&w_2(x,t)e^{-\I Bx}=
				e^{2\I{B^2} t}
				+\int_{-\infty}^x
				\left(
				Ae^{\I By+4\I{B^2} t}
				-e^{-\I By}\bar{q}(-y,t)
				\right)
				w_1(y,t)\,dy\\
				&\qquad\qquad\qquad\,\,\,
				-A\int_{-\infty}^x(x-y)
				e^{\I By}q(y,t)w_2(y,t)\,dy.
			\end{split}
		\end{equation}
		Combining \eqref{v1v2b} and \eqref{w1w2}, we arrive at
		\begin{equation*}
			\begin{pmatrix}
				w_1(x,t)\\w_2(x,t)
			\end{pmatrix}=
			\frac{2\I}{A}
			\begin{pmatrix}
				v_1(x,t)\\v_2(x,t)
			\end{pmatrix},
		\end{equation*}
		which, together with \eqref{Psi1b} and \eqref{Psi2b}, implies \eqref{Psi12B}.

		Finally, to prove item (6). 
		Observe that
		the matrices in the right-hand side of \eqref{LP} are traceless. Thus, using the Liouville theorem, we conclude that
		$$
		\det\Psi_j(x,t,k)=
		\lim\limits_{y\to(-1)^j\infty}
		\det\Psi_j(y,t,k),\,\,
		x,t\in\mathbb{R},\,
		k\in\mathbb{R}\setminus\{B,-B\},\,
		j=1,2.
		$$
		Using expansions \eqref{Psi12B}, we have that
		$\det\Psi_j(x,t,k)=\mathcal{O}(1)$, 
		as $k\to\pm B$, $j=1,2$.
		Thus, we establish item (6) for all $k\in\mathbb{R}$.
	\end{proof}
	
	\subsection{Scattering data}
	Since the functions $\Phi_1(x,t,k)$ and $\Phi_2(x,t,k)$ satisfy both equations in system \eqref{LP}, we have the following relation for $\Phi_1$ and $\Phi_2$:
	\begin{equation}
		\label{S}
		\Phi_1(x,t,k)=\Phi_2(x,t,k)S(k),\quad 
		k\in\mathbb{R}\setminus\{B,-B\},
	\end{equation}
	for all $x,t\in\mathbb{R}$.
	%Here the $2\times 2$ matrix $S(k)$ is called the scattering matrix.
	Taking into account items (1) and (4) in Proposition \ref{PPsi}, we conclude that $\Phi_j$, $j=1,2$, satisfy the following symmetry relation:
	\begin{equation}\label{phi-sym}
		\sigma_1
		\overline{\Phi_1}(-x,t,-k)
		\sigma_1^{-1}=\Phi_2(x,t,k),\quad k\in\mathbb{R}\setminus\{B,-B\}.
	\end{equation}
	Using \eqref{phi-sym}, we can write
	the scattering matrix $S(k)$ in the following form:
	\begin{equation}\label{Ssp}
		S(k)=
		\begin{pmatrix}
			a_1(k)& -\bar{b}(-k)\\
			b(k)& a_2(k)
		\end{pmatrix},\quad k\in\mathbb{R}\setminus\{B,-B\},
	\end{equation}
	with some $b(k)$, $a_1(k)$ and $a_2(k)$.
	
	%Observe that items (1) and (6) of Proposition \ref{PPsi} imply that $\det\Phi_j(x,t,k)=1$.
	%Therefore \eqref{S} implies the determinant relation, which reads:
	
	Relation \eqref{S} yields that the spectral functions $a_j(k)$, $j=1,2$, and $b(k)$ can be defined in terms of the following determinants:
	%in terms of the known initial data $q_0(x)$ only, by the following determinant relations
	\begin{subequations}\label{sd}
		\begin{align}
			\label{sda}
			&a_1(k)=\det\left(
			\Psi_1^{(1)}(0,0,k),\Psi_2^{(2)}(0,0,k)
			\right),
			\quad k\in\overline{\C^+}
			\setminus\{B,-B\},\\
			&a_2(k)=\det\left(
			\Psi_2^{(1)}(0,0,k),\Psi_1^{(2)}(0,0,k)
			\right),
			\quad k\in\overline{\C^-},\\
			\label{sdb}
			&b(k)=\det\left(
			\Psi_2^{(1)}(0,0,k),\Psi_1^{(1)}(0,0,k)
			\right),
			\quad k\in\R\setminus\{B\},
		\end{align}
	\end{subequations}
	where we have used that $\Psi_j(0,0,k)=\Phi_j(0,0,k)$.
	Notice that \eqref{sd} implies that the spectral data can be found in terms of the known initial data $q_0$ only.
	In view of \eqref{Psi12B}, we have the following expansions of $a_1(k)$ and $b(k)$ as $k\to B$ and/or $k\to-B$:
	\begin{equation}\label{a1pmB}
		a_1(k)=\frac{a_1^{\pm B}}{k\mp B}+\mathcal{O}(1),\quad k\to\pm B,\quad
		b(k)=\frac{b^B}{k-B}+\mathcal{O}(1),\quad
		k\to B,
	\end{equation}
	with some
	$a_1^{\pm B}, b^B\in\mathbb{C}$.
	
	Leveraging determinant relations \eqref{sd} and the properties of $\Psi_j$ given in Proposition \ref{PPsi}, we can obtain important analytical, continuity and symmetry properties of the spectral functions, which are summarized in the proposition given below.
	\begin{proposition}\label{prsp}
		The spectral functions $a_j(k)$, $j=1,2$, and $b(k)$ satisfy the following properties:
		\begin{enumerate}
			\item analyticity:
			$a_{1}(k)$ is analytic in  $k\in\mathbb{C}^{+}$ and continuous in 
			$\overline{\C^+}\setminus\{B,-B\}$, while
			$a_{2}(k)$ is analytic in 
			$k\in\mathbb{C}^{-}$
			and continuous in 
			$\overline{\C^-}$;
			
			\item behavior for the large $k$:
			$a_{1}(k)=1+\mathcal{O}
			\left(k^{-1}\right)$,
			$k\to\infty$, 
			$k\in\overline{\C^+}$, 
			$a_{2}(k)=1+\mathcal{O}
			\left(k^{-1}\right)$,
			$k\to\infty$, 
			$k\in\overline{\C^-}$
			and 
			$b(k)=\mathcal{O}
			\left(k^{-1}\right)$, $k\to\infty$, $k\in\mathbb{R}$;
			
			\item
			symmetries:
			$\overline{a_{1}}\left(-\bar{k}\right)=a_1(k)$,  
			$k\in\overline{\C^+}
			\setminus\{B,-B\}$ and
			$\overline{a_{2}}\left(-\bar{k}\right)=a_2(k)$,  
			$k\in\overline{\C^-}$;
			
			\item
			determinant relation:
			$a_{1}(k)a_{2}(k)
			+b(k)\bar{b}(-k)=1$, 
			$k\in\R\setminus\{B,-B\}$.
			
			\item the coefficients in \eqref{a1pmB} satisfy the following relations:
			\begin{equation}\label{scoeff}
				a_1^B=-\frac{A}{2\I}\bar{b}(-B),
				\quad
				a_1^{-B}=-\overline{a_1^B},
				\quad
				a_1^B\left(
				a_2(B)-\frac{2\I}{A}b^B\right)=0.
			\end{equation}
		\end{enumerate}
	\end{proposition}
	\begin{proof}
		Item (1) follows from the analytical properties of the columns 
		$\Psi_i^{(j)}$, $i,j=1,2$, see items (2) and (3) in Proposition \ref{PPsi}.
		Asymptotic behavior of $\Psi_i^{(j)}$, $i,j=1,2$, given in \eqref{Psas2}--\eqref{Psas1}, yields item (2).
		Item (3) and (4) follows from \eqref{Psymm} and item (6) in Proposition \ref{PPsi}, respectively.
		
		Let us prove item (5).
		Symmetry relation 
		$\overline{a_{1}}
		\left(-\bar{k}\right)=a_1(k)$
		implies that 
		$a_1^{-B}=-\overline{a_1^B}$.
		From the determinant relation, see item (4), we conclude that
		\begin{equation}\label{dz}
			a_2(B)a_1^B+\overline{b}(-B)b^B=0.
		\end{equation}
		
		Combining \eqref{sda}, \eqref{sdb} and \eqref{Psi12B}, we obtain that
		\begin{equation*}
			\begin{split}
				&\frac{a_1^{-B}}{k+B}+\mathcal{O}(1)
				=\det
				\begin{pmatrix}
					\left(\Psi_1\right)_{1,1}(0,0,-B)
					&\frac{-\overline{v_2}(0,0)}{k+B}\\
					\left(\Psi_1\right)_{2,1}(0,0,-B)
					&\frac{-\overline{v_1}(0,0)}{k+B}
				\end{pmatrix}
				+\mathcal{O}(1),\\
				&b(-B)
				=\det
				\begin{pmatrix}
					\frac{-2\I}{A}\overline{v_2}(0,0)
					&\left(\Psi_1\right)_{1,1}(0,0,-B)\\
					\frac{-2\I}{A}\overline{v_1}(0,0)
					&\left(\Psi_1\right)_{2,1}(0,0,-B)
				\end{pmatrix},
			\end{split}
		\end{equation*}
		which implies that
		$
		a_1^{-B}=-\frac{A}{2\I}b(-B)
		$.
		Recalling \eqref{dz} and that
		$a_1^{-B}=-\overline{a_1^B}$, we arrive at \eqref{scoeff}.
	\end{proof}
	
	\subsection{Pure step initial data}
	In this subsection we explicitly calculate the spectral functions $a_j(k)$, $j=1,2$, and $b(k)$ for the ``shifted step'' initial data $q_{0,R}(x)$ given in \eqref{shst}.
	\begin{proposition}
		The spectral functions associated to the step-like initial data given in \eqref{shst} have the following form:
		\begin{equation}\label{spss}
			a_1(k)=1+\frac{A^2e^{4\I kR}}
			{4(k^2-B^2)},\quad
			a_2(k)=1,\quad
			b(k)=\frac{-\I Ae^{2\I R(k-B)}}{2(k-B)}.
		\end{equation}
	\end{proposition}
	\begin{proof}
		Combining \eqref{PhPs} and \eqref{S} we obtain the following representation for $S(k)$:
		\begin{equation}\label{S-R}
			S(k)=e^{-\I(k+B)R\sigma_3}
			\Psi_2^{-1}(-R,0,k)
			\Psi_1(-R,0,k)
			e^{\I(k-B)R\sigma_3}.
		\end{equation}
		Taking into account that the matrix $U(x,0)$
		with $q(x,0)=q_{0,R}(x)$ has the form (see \eqref{U}, \eqref{shst})
		\begin{equation*}
			U(x,0)=
			\begin{cases}
				\begin{pmatrix}
					0&0\\
					-Ae^{2\I Bx}&0
				\end{pmatrix},&x<-R,\\
				%\begin{pmatrix}
				%0&0\\
				%0&0
				%\end{pmatrix},
				\mathbf{0}_{2\times2},
				&-R\leq x\leq R,\\
				\begin{pmatrix}
					0&Ae^{2\I Bx}\\
					0&0
				\end{pmatrix},&x>R,
			\end{cases}
		\end{equation*}
		we have from \eqref{Psi1} that
		\begin{equation}\label{Psi1-R}
			\Psi_1(-R,0,k)=e^{\I BR\sigma_3}N_-(k).
		\end{equation}
		Equation \eqref{Psi2} implies that $\Psi_2(x,0,k)$ satisfies the following integral equation for
		$-R\leq x\leq R$:
		\begin{equation*}
			\Psi_2(x,0,k)=e^{\I Bx\sigma_3}N_+(k)
			+\int_R^x
			\begin{pmatrix}
				0&-Ae^{-\I k(x-y)+2\I By}\\
				0&0
			\end{pmatrix}
			\Psi_2(y,0,k)
			e^{\I(k+B)(x-y)\sigma_3}\,dy,
		\end{equation*}
		which can be solved explicitly as follows
		\begin{equation}\label{Psi2-R}
			\Psi_2(x,0,k)=\begin{pmatrix}
				e^{\I Bx}&
				\frac{-\I Ae^{\I Bx}}{2(k+B)}
				e^{2\I(k+B)(R-x)}\\
				0& e^{-\I Bx}
			\end{pmatrix},\quad
			-R\leq x\leq R.
		\end{equation}
		Combining \eqref{S-R}, \eqref{Psi1-R}, \eqref{Psi2-R} and recalling notation \eqref{Ssp}, we arrive at \eqref{spss}.
	\end{proof}
	
	%Now let us study the spectral singularities of $a_1(k)$ given in \eqref{spss}.
	%Namely, we are interested in zeros for $k\in\mathbb{C}^+\cup\mathbb{R}$ and winding of the argument along $\mathbb{R}$.
	Now let us study zeros of $a_1(k)$ for $k\in\mathbb{C}^+\cup\mathbb{R}$ and the winding of the argument of $a_1(k)$ along $\mathbb{R}$ in the case of the pure step initial data.
	These properties will motivate the formulation of the associated (basic) RH problem for the Cauchy problem \eqref{IVP} with general initial data.
	\begin{proposition}[Zeros of $a_1(k)$]
		\label{a1Zs}
		The function $a_1(k)$, given in \eqref{spss}, can have the following zeros for $k\in\mathbb{C}^+\cup\mathbb{R}$:
		%\item zeros of $a_1(k)$ for
		%$k\in\mathbb{C}^+\cup\mathbb{R}\setminus\{B,-B\}$:
		\begin{enumerate}[1)]
			\item real zeros of $a_1(k)$:
			\begin{enumerate}[{1.}1)]
				\item if there exists 
				$n\in\N$ such that
				$\pi^2n^2=R^2(4B^2-A^2)$, then $a_1(k)$ has two simple real zeros, which read
				$$
				k=\pm\frac{\pi n}{2R}
				=\pm\frac{1}{2}\left(4B^2-A^2\right)^{1/2};
				$$
				\item if there exists 
				$n\in\Z$ such that
				$\pi^2(1/2+n)^2=R^2(4B^2+A^2)$, then $a_1(k)$ has two simple real zeros, which read
				$$
				k=\frac{\pi (1+2n)}{4R}
				=\pm\frac{1}{2}\left(4B^2+A^2\right)^{1/2};
				$$
				\item if $4B^2=A^2$, then $a_1(k)$ has one real zero of multiplicity two at $k=0$;
			\end{enumerate}
			\medskip
			\item purely imaginary zeros of $a_1(k)$ for $\Imm\,k>0$:
			\begin{enumerate}[{2.}1)]
				\item if $4B^2-A^2<0$, then 
				$a_1(k)$ has one simple zero $k=\I k_0$, where $k_0>0$ is a unique solution of the following transcendental equation:
				\begin{equation}\label{tre}
					A^2e^{-4kR}=4(B^2+k^2),\quad k>0,
					\quad R\geq0;
				\end{equation}
				\item if $4B^2-A^2\geq 0$, then $a_1(k)$ has no imaginary zeros for $\Imm\,k>0$;
			\end{enumerate}
			\medskip
			\item assume that $0\leq 4|B|R\leq\pi$.
			Then depending on the values of $A^2, B^2$ and $R$,
			$a_1(k)$ has the following zeros for $k\in\C^+$, 
			$\Imm\,k\neq0$:
			\begin{enumerate}[{3.}1)]
				\item if 
				$0\leq R\leq\frac{\pi}
				{2\left(4B^2+A^2\right)^{1/2}}$,
				then $a_1(k)$ has no zeros for $k\in\C^+$, 
				$\Imm\,k\neq0$;
				\item if
				there exists $n\in\N$ such that
				$\frac{(2n-1)\pi}
				{2\left(4B^2+A^2\right)^{1/2}}< R
				\leq\frac{(2n+1)\pi}
				{2\left(4B^2+A^2\right)^{1/2}}$, then $a_1(k)$ has $2n$ simple zeros
				$\{p_j,-\overline{p}_j\}_{j=1}^n$,
				such that
				$$
				\Ree\,p_j\in\left(-\frac{j\pi}{2R},
				\frac{(1-2 j)\pi}{4R}\right),\quad
				j=1,\dots,n.
				$$
				Here $\Ree\,p_j=\frac{\tau_j}{4R}$ and
				$\Imm\,p_j=\frac{y_j}{4R}$, $j=1,\dots,n$, where
				$(\tau_j,y_j)$ is a unique solution of transcendental equation \eqref{sety1} considered for 
				$\tau\in(2\pi j-\pi,2\pi j)$.
			\end{enumerate}
		\end{enumerate}
	\end{proposition}
	\begin{proof}
		The details of the proof can be found in Appendix \ref{Ap1}.		
	\end{proof}
	\begin{remark}
		Notice that the properties of zeros of $a_1(k)$, obtained in Proposition \ref{a1Zs}, are consistent, in the limit $B\to0$, with \cite[Proposition 2]{RS21-CMP}, where it was considered the ``shifted step'' \eqref{shst} with $B=0$.
	\end{remark}
	\begin{proposition}[Winding of the argument of $a_1(k)$]
		\label{a1Ws}
		Assume that $0<4|B|R<\pi$.
		Then the winding of the argument of the function $a_1(k)$ for 
		$k\in(-\infty,0)\setminus\{-|B|\}$, given in 
		\eqref{spss}, has the following properties:
		\begin{enumerate}[1)]
			\item if $0<R<\frac{\pi}
			{2\left(4B^2+A^2\right)^{1/2}}$, then
			\begin{equation}\label{R0}
				\begin{split}
					&\int_{-\infty}^{k}d\arg(a_1(s))\in(-\pi,\pi),
					\quad k\in(-\infty,-|B|),\\
					&\lim\limits_{k\uparrow-|B|}
					\int_{-\infty}^{k}d\arg(a_1(s))
					=-\theta_{-|B|},\quad\theta_{-|B|}\in(0,\pi),
				\end{split}
			\end{equation}
			and
			\begin{equation}\label{R0-1}
				\begin{split}
					&\lim\limits_{k\downarrow-|B|}
					\arg(a_1(k))=\pi-\theta_{-|B|},\\
					&\int_{-|B|}^{k}d\arg(a_1(s))
					\in(0,\pi),
					\quad k\in(-|B|,0);
				\end{split}
			\end{equation}
			
			\item if $\frac{(2n-1)\pi}
			{2\left(4B^2+A^2\right)^{1/2}}<R<\frac{(2n+1)\pi}
			{2\left(4B^2+A^2\right)^{1/2}}$ for some $n\in\N$, then, using notations
			\begin{equation}\label{om1}
				\omega_{n+1}=-\infty,\quad
				\omega_j=\frac{-(2j-1)\pi}{4R},\,j=1,\dots,n,\quad
				\omega_0=-|B|,
			\end{equation}
			we have
			\begin{equation}\label{Rn}
				\begin{split}
					&\int_{-\infty}^{\omega_{n-j+1}}
					\,d\arg a_1(k)=(2j-1)\pi,\quad j=1,\dots,n,\\
					%\int_{-\infty}^k\,d\arg a_1(k)
					%\in(-\pi,\pi),\quad k<\omega_n,\\
					&\int_{-\infty}^k\,d\arg a_1(k)
					\in(2\pi j-\pi,2\pi j+\pi),\quad 
					k\in(\omega_{n-j+1},\omega_{n-j}),\quad
					j=0,\dots n,\\
					%&\int_{-\infty}^k\,d\arg a_1(k)
					%\in(2\pi n-\pi,2\pi n),\quad 
					%k\in(\omega_{1},-|B|),\\
					&\lim\limits_{k\uparrow-|B|}
					\int_{-\infty}^{k}d\arg(a_1(s))
					=2\pi n-\theta_{-|B|},\quad\theta_{-|B|}\in(0,\pi),
				\end{split}
			\end{equation}
			and
			\begin{equation}\label{Rn-1}
				\begin{split}
					&\lim\limits_{k\downarrow-|B|}
					\arg(a_1(k))=
					2\pi n+\pi-\theta_{-|B|},\\
					&\int_{-|B|}^{k}d\arg(a_1(s))
					\in(2\pi n,2\pi n+\pi),
					\quad k\in(-|B|,0);
				\end{split}
			\end{equation}
			\item 
			%taking the $\arg a_1(k)$ for $k\downarrow-|B|$ as in items 1) and 2) above, 
			depending on the sign of 
			$\left(4B^2-A^2\right)$
			and the value of $R$, we have the following winding of the argument at $k=0$:
			\begin{subequations}
				\label{arga10}
				\begin{align}
					\label{arga10-a}
					&\mbox{if }4B^2-A^2>0,\mbox{ then }\\
					\nonumber
					&\qquad\qquad
					\int_{-|B|}^{0}d\arg(a_1(k))=
					\begin{cases}
						0,&\mbox{}
						0<R<\frac{\pi}
						{2\left(4B^2+A^2\right)^{1/2}},\\
						2\pi n,&\mbox{}
						\frac{(2n-1)\pi}
						{2\left(4B^2+A^2\right)^{1/2}}
						<R<\frac{(2n+1)\pi}
						{2\left(4B^2+A^2\right)^{1/2}},
					\end{cases}\\
					\label{arga10-b}
					&\mbox{if }4B^2-A^2<0,\mbox{ then }\\
					\nonumber
					&\qquad\qquad
					\int_{-|B|}^{0}d\arg(a_1(k))=
					\begin{cases}
						\pi,&\mbox{}
						0<R<\frac{\pi}
						{2\left(4B^2+A^2\right)^{1/2}},\\
						2\pi n+\pi,&\mbox{}
						\frac{(2n-1)\pi}
						{2\left(4B^2+A^2\right)^{1/2}}
						<R<\frac{(2n+1)\pi}
						{2\left(4B^2+A^2\right)^{1/2}}.
					\end{cases}
				\end{align}
			\end{subequations}
		\end{enumerate}
	\end{proposition}
	\begin{proof}
		The proof is given in Appendix \ref{Ap2}.	
	\end{proof}
	\begin{remark}
		Notice that	we choose the argument of $a_1(k)$ as $k\downarrow-|B|$
		in \eqref{R0-1} and \eqref{Rn-1} in such a way that $\arg a_1(k)$ is continuous for
		$k\in\{0<|k+|B||<\ve\}
		\cap\overline{\C^+}$.
	\end{remark}
	
	Motivated by the properties of the spectral functions for the pure step initial data \eqref{shst}, in what follows we impose the restrictions for $a_j(k)$, $j=1,2$ described below:
	
	\begin{description}
		\item [Assumptions A]
		\indent
		%\textbf{Assumptions A:}
		\begin{enumerate}[{A.}1)]
			\item  $a_1^{\pm B}\ne 0$, 
			%in the expansion of $a_1(k)$ as $k\to\pm B$, 
			see \eqref{a1pmB};
			\item $a_1(k)\neq 0$ for all $k\in\R\setminus\{B,-B\}$ and
			$a_2(k)\neq0$ for all $k\in\R$;
			\item $a_2(k)$ has no zeros for 
			$k\in\C^-$;
			\item $b(k)$ has no zeros for $k\in\R$;
			\item Fix the argument of the product $a_1(k)a_2(k)$ at $k=-\infty$ as follows:
			\begin{equation}\label{arga}
				\lim\limits_{k\to-\infty}
				\arg(a_1(k)a_2(k))=0
			\end{equation}
			and assume that one of two cases below holds:
		\end{enumerate}
	\end{description}

	\begin{description}
		\item [Case I]
		\indent
		%\textbf{Case I:}
		\begin{enumerate}[{I.}1)]
			\item $a_1(k)$ has $2n+1$, 
			$n\in\N\cup\{0\}$, simple zeros in $\C^+$ at
			$k\in\left\{
			\I k_0,
			\{p_j,-\overline{p}_j\}_{j=1}^n
			\right\}$, where
			$$
			k_0>0,\quad
			\Imm\,p_j>0,\,j=1,\dots,n,\quad
			\Ree\,p_n<\dots<
			\Ree\,p_1<0;
			$$
			\item there exist $\omega_j<0$,
			$j=0,\dots,n+1$, such that
			\begin{equation}\label{ordom}
				-\infty=\omega_{n+1}
				<\Ree\,p_n<\omega_n<\Ree\,p_{n-1}
				<\omega_{n-1}<\dots<\Ree\,p_1
				<\omega_1<\omega_0=-|B|,
			\end{equation}
			and (cf.\,\,\eqref{Rn})
			\begin{equation}\label{gRn}
				\begin{split}
					&\int_{-\infty}^{\omega_{n-j+1}}
					\,d\arg\left(a_1(k)a_2(k)\right)
					=(2j-1)\pi,
					\quad j=1,\dots,n,\\
					&\int_{-\infty}^k\,
					d\arg\left(a_1(k)a_2(k)\right)
					\in(2\pi j-\pi,2\pi j+\pi),
					\quad 
					k\in(\omega_{n-j+1},\omega_{n-j}),
					\quad j=0,\dots n,\\
					&\lim\limits_{k\uparrow-|B|}
					\int_{-\infty}^{k}d\arg(a_1(s)a_2(s))
					=2\pi n-\theta_{-|B|},
					\quad\theta_{-|B|}
					\in(0,\pi),
				\end{split}
			\end{equation}
			as well as (cf.\,\,\eqref{Rn-1}
			and \eqref{arga10-b})
			\begin{equation}\label{gRn-1}
				\begin{split}
					&\lim\limits_{k\downarrow-|B|}
					\arg(a_1(k)a_2(k))=
					2\pi n+\pi-\theta_{-|B|},\\
					&\int_{-|B|}^{k}d\arg(a_1(s)a_2(s))
					\in(2\pi n,2\pi n+\pi),
					\quad k\in(-|B|,0),
				\end{split}
			\end{equation}
			and
			\begin{equation}\label{aa1g}
				\int_{-|B|}^{0}d\arg(a_1(k)a_2(k))
				=2\pi n+\pi.
			\end{equation}
		\end{enumerate}
		
		\item[Case II]
		%\textbf{Case II:}
		\indent
		\begin{enumerate}[{II.}1)]
			\item $a_1(k)$ has $2n$, 
			$n\in\N\cup\{0\}$, simple zeros in $\C^+$ at
			$k\in\left\{
			\{p_j,-\overline{p}_j\}_{j=1}^n
			\right\}$, where
			$$
			\Imm\,p_j>0,\,j=1,\dots,n,\quad
			\Ree\,p_n<\dots<
			\Ree\,p_1<0;
			$$
			\item there exist $\omega_j<0$,
			$j=0,\dots,n+1$, such that
			\eqref{ordom}, \eqref{gRn}, and
			\eqref{gRn-1} hold and 
			(cf.\,\,\eqref{aa1g})
			\begin{equation}\label{aa1g1}
				\int_{-|B|}^{0}d\arg(a_1(k)a_2(k))
				=2\pi n.
			\end{equation}
		\end{enumerate}
	\end{description}
	\begin{remark}
		Consider the spectral functions $a_1(k)$ and $a_2(k)$ given in \eqref{spss}, which correspond to ``shifted step'' initial data \eqref{shst}.
		Then Propositions \ref{a1Zs} and \ref{a1Ws} imply that
		\begin{enumerate}[(i)]
			\item if $0<R<\frac{\pi}
			{2\left(4B^2+A^2\right)^{1/2}}$,
			$0< 4|B|R\leq\pi$, and
			$4B^2-A^2<0$,
			then $a_j(k)$, $j=1,2$, satisfy
			Assumption A and Case I with $n=0$;
			\item if 
			$\frac{(2n-1)\pi}
			{2\left(4B^2+A^2\right)^{1/2}}
			<R<\frac{(2n+1)\pi}
			{2\left(4B^2+A^2\right)^{1/2}}$ for some $n\in\N$, $0< 4|B|R\leq\pi$, and $4B^2-A^2<0$,
			then $a_j(k)$, $j=1,2$ satisfy
			Assumption A and Case I with the same $n$;
			\item if $0<R<\frac{\pi}
			{2\left(4B^2+A^2\right)^{1/2}}$,
			$0< 4|B|R\leq\pi$,
			$4B^2-A^2>0$,
			and
			$\pi^{-2}R^2(4B^2-A^2)$ is not a squared integer,
			then $a_j(k)$, $j=1,2$ satisfy
			Assumption A and Case II with $n=0$;
			\item if $\frac{(2n-1)\pi}
			{2\left(4B^2+A^2\right)^{1/2}}
			<R<\frac{(2n+1)\pi}
			{2\left(4B^2+A^2\right)^{1/2}}$ for some $n\in\N$,
			$0< 4|B|R\leq\pi$, 
			$4B^2-A^2>0$, and
			$\pi^{-2}R^2(4B^2-A^2)$ is not a squared integer,
			then $a_j(k)$, $j=1,2$, satisfy
			Assumption A and Case II with the same $n$.
		\end{enumerate}
	\end{remark}
	
	\subsection{Basic Riemann-Hilbert problem}
	%We are going to define the sectionally meromorphic matrix-valued function, which satisfies a jump condition along a some contour in the complex plane and is normalized at infinity.
	
	Define a sectionally meromorphic matrix-valued function $M(x,t,k)$ as follows:
	
	\begin{equation}
		\label{DM}
		M(x,t,k)=
		\begin{cases}
			\left(
			\frac{\Psi_1^{(1)}(x,t,k)}{a_{1}(k)},
			\Psi_2^{(2)}(x,t,k)\right),&
			k\in\C^+,\\
			\left(\Psi_2^{(1)}(x,t,k),
			\frac{\Psi_1^{(2)}(x,t,k)}{a_{2}(k)}
			\right),& k\in\C^-.\\
		\end{cases}
	\end{equation}
	Combining \eqref{DM} and scattering relation (\ref{S}), we conclude that the boundary values of the matrix $M(x,t,k)$, relative to the contour $\mathbb{R}\setminus\{-B\}$, satisfy the following multiplicative jump condition (notice that \eqref{Psi12B} and 
	\eqref{a1pmB} imply that $M_\pm$ is bounded at $k=B$):
	\begin{equation}\label{jr}
		M_+(x,t,k)=M_-(x,t,k)J(x,t,k),\quad k\in\R\setminus\{-B\},
	\end{equation}
	where
	\begin{equation}\label{jump}
		J(x,t,k)=
		\begin{pmatrix}
			1+r_{1}(k)r_{2}(k)& r_{2}(k)e^{-2\I kx-4\I k^2t}\\
			r_1(k)e^{2\I kx+4\I k^2t}& 1
		\end{pmatrix}
		,\quad k\in\R\setminus\{-B\},
	\end{equation}
	with the reflection coefficients $r_1(k)$ and $r_2(k)$ given by (see also \eqref{rj-s} below)
	$$
	r_1(k)=\frac{b(k)}{a_1(k)},\quad
	r_2(k)=\frac{\bar{b}(-k)}{a_2(k)},\quad
	k\in\R\setminus\{-B\}.
	$$
	Taking into account Assumptions 
	A.1)--A.3), we have (see \eqref{a1pmB})
	\begin{equation}\label{rj-s}
		\begin{split}
			&r_1(k)=\frac{b^B}{a_1^B}+\mathcal{O}(k-B),\quad
			k\to B,\quad
			r_1(k)=\frac{b(-B)}{a_1^{-B}}(k+B)
			+\mathcal{O}(k+B)^2,\quad k\to -B,\\
			&r_2(k)=\frac{\bar{b}(-B)}{a_2(B)}
			+\mathcal{O}(k-B),\quad
			k\to B,\quad
			r_2(k)=\frac{-\overline{b^B}}{a_2(-B)(k+B)}+
			\mathcal{O}(1),\quad k\to -B.
		\end{split}
	\end{equation}
	Observe that the symmetries described in Proposition \ref{prsp}, item (3) and \eqref{a1pmB} imply that 
	%(notice that in view of \eqref{} and Assumptions A.2)--A.3), the functions $a_j^{-1}(k)$)
	\begin{equation}
		\label{r-sym}
		r_1(-k) r_2(-k)=\overline{r_1}(k)\;  \overline{r_2}(k), \quad k\in\R,
	\end{equation}
	while the determinant relation, see Proposition \ref{prsp}, item (4), yields
	\begin{equation}\label{r-a}
		1+r_1(k) r_2(k)=\frac{1}{a_1(k)a_2(k)},\quad k\in\R.
	\end{equation}
	
	Behavior of $\Psi_j$ and $a_j$, $j=1,2$, for the large $k$, see \eqref{Psas1}--\eqref{Psas2} and Proposition \ref{prsp}, item (2), implies the following normalization condition:
	\begin{equation}
		M(x,t,k)=I
		+\mathcal{O}\left(k^{-1}\right),
		\quad k\to\infty.
	\end{equation}
	%where here and below $I$ is the $2\times2$ identity matrix.
	
	The matrix function $M(x,t,k)$ has the following residue conditions at the zeros of $a_1(k)$:
	
	\textbf{Case I:}
	\begin{equation}
		\label{resinI}
		\begin{split}
			\underset{k=\I k_0}
			{\operatorname{Res}} M^{(1)}(x,t,k)&=
			\frac{\gamma_0}{\dot{a}_1(\I k_0)}
			e^{-2k_0x-4\I k_0^2t}
			M^{(2)}(x,t,\I k_0),\\
			\underset{k=p_j}
			{\operatorname{Res}} M^{(1)}(x,t,k)&=
			\frac{\eta_j}{\dot a_1(p_j)}
			e^{2\I p_jx+4\I p_j^2t}
			M^{(2)}(x,t,p_j),
			\quad j=1,\dots,n,\\
			\underset{k=-\overline{p}_j}
			{\operatorname{Res}}
			M^{(1)}(x,t,k)&=
			\frac{\hat{\eta}_j}
			{\dot a_1(-\overline{p}_j)}
			e^{-2\I\overline{p}_jx
				+4\I\overline{p}_j^2t}
			M^{(2)}(x,t,-\overline{p}_j),
			\quad j=1,\dots,n.
		\end{split}
	\end{equation}
	
	\textbf{Case II:}
	\begin{equation}
		\label{resinII}
		\begin{split}
			\underset{k=p_j}
			{\operatorname{Res}} M^{(1)}(x,t,k)&=
			\frac{\eta_j}{\dot a_1(p_j)}
			e^{2\I p_jx+4\I p_j^2t}
			M^{(2)}(x,t,p_j),
			\quad j=1,\dots,n,\\
			\underset{k=-\overline{p}_j}
			{\operatorname{Res}}
			M^{(1)}(x,t,k)&=
			\frac{\hat{\eta}_j}
			{\dot a_1(-\overline{p}_j)}
			e^{-2\I\overline{p}_jx
				+4\I\overline{p}_j^2t}
			M^{(2)}(x,t,-\overline{p}_j),
			\quad j=1,\dots,n.
		\end{split}
	\end{equation}
	Here $\gamma_0$, $\eta_j$ and $\hat\eta_j$ are defined in terms of the initial data as constants of proportionality of the following columns:
	\begin{equation*}
		\begin{split}
			&\Psi_1^{(1)}(0,0,\I k_0)
			=\gamma_0\Psi_2^{(2)}(0,0,\I k_0),\\ &\Psi_1^{(1)}(0,0,p_j)
			=\eta_j\Psi_2^{(2)}(0,0,p_j),\quad
			\Psi_1^{(1)}(0,0,p_j)
			=\hat\eta_j
			\Psi_2^{(2)}(0,0,-\overline{p}_j),
			\quad j=1,\dots,n.
		\end{split}
	\end{equation*}
	Then symmetry relation \eqref{Psymm} implies that
	\begin{equation*}
		|\gamma_0|=1,\quad
		\hat\eta_j=
		\frac{1}{\overline{\eta}_j},
		\quad j=1,\dots,n.
	\end{equation*}
	
	%Since both  functions $\Psi_j(x,t,k)$, $j=1,2$ and the spectral function $a_1(k)$ have singularities at  $k=\pm B$ (see Propositions \ref{PPsi} and \ref{prsp}), it is necessary to describe the behavior of $M(x,t,k)$ at these points.
	%Direct calculations show that  $M(x,t,k)$ is bounded at $k=B$, while it has the following singularities as $k$ approaches $-B$ from different complex half-planes
	%(recall \eqref{a1pmB}):
	%\begin{equation}
	%\begin{split}
	%& M_+(x,t,k)=
	%\begin{pmatrix}
	%\frac{\Psi_{1(11)}(x,t,-B)}{a_{1}^{-B}}& -\overline{v_2}(-x,t)\\
	%\frac{\Psi_{1(21)}(x,t,-B)}{a_{1}^{-B}}& -\overline{v_1}(-x,t)
	%\end{pmatrix}
	%(I+\mathcal{O}(k+B))
	%\begin{pmatrix}
	%k+B& 0\\
	%0& \frac{1}{k+B}
	%\end{pmatrix}
	%,\,k\rightarrow -B+\I 0,\\
	%& M_-(x,t,k)=
	%\begin{pmatrix}
	%-\frac{2\I}{A}\overline{v_2}(-x,t)
	%&\frac{\Psi_{1(12)}(x,t,-B)}{a_2(-B)}\\
	%-\frac{2\I}{A}\overline{v_1}(-x,t)
	%&\frac{\Psi_{1(22)}(x,t,-B)}{a_2(-B)}
	%\end{pmatrix}
	%+\mathcal{O}(k+B)
	%,\,k\rightarrow -B-\I 0.
	%\end{split}
	%\end{equation}
	
	Finally, $M(x,t,k)$ has the following singularity as $k$ approaches $-B$ from the upper half-plane:
	\begin{equation*}
		\begin{split}
			&M(x,t,k)
			\begin{pmatrix}
				(k+B)^{-1}& 0\\
				0& k+B
			\end{pmatrix}
			=\mathcal{O}(1),\quad
			k\to-B+\I 0,\\
			&M(x,t,k)=\mathcal{O}(1),\quad
			k\to-B-\I 0.
		\end{split}
	\end{equation*}
	
	The solution of the Cauchy problem
	\eqref{IVP} can be retrieved in terms of $M(x,t,k)$ as follows:
	\begin{equation}\label{sol}
		q(x,t)=2\I\lim_{k\to\infty}
		kM_{1,2}(x,t,k),
	\end{equation}
	and
	\begin{equation}\label{sol1}
		q(-x,t)=-2\I\lim_{k\to\infty}
		k\overline{M_{2,1}}(x,t,k).
	\end{equation}
	
	\section{Long-time asymptotics}
	
	\subsection{Phase function and jump factorizations}
	Introduce the variable $\xi:=\frac{x}{4t}$ and the phase function
	$\theta$ as follows:
	%(notice that the phase function is the same as in the conventional NLS equation):
	\begin{equation}\label{theta}
		\theta(k,\xi)=4k\xi+2k^2.
	\end{equation}
	Then the jump matrix $J(x,t,k)$, see \eqref{jump}, can be written in the form of triangular factorizations:
	\begin{subequations}\label{tr}
		\begin{align}
			\label{tr1}
			J(x,t,k)&=
			\begin{pmatrix}
				1& 0\\
				\frac{r_1(k)}{1+ r_1(k)r_2(k)}e^{2\I t\theta}& 1\\
			\end{pmatrix}
			\begin{pmatrix}
				1+ r_1(k)r_2(k)& 0\\
				0& \frac{1}{1+ r_1(k)r_2(k)}\\
			\end{pmatrix}
			\begin{pmatrix}
				1& \frac{ r_2(k)}{1+ r_1(k)r_2(k)}e^{-2\I t\theta}\\
				0& 1\\
			\end{pmatrix}
			\\
			\label{tr2}
			&=
			\begin{pmatrix}
				1& r_2(k)e^{-2\I t\theta}\\
				0& 1\\
			\end{pmatrix}
			\begin{pmatrix}
				1& 0\\
				r_1(k)e^{2\I t\theta}& 1\\
			\end{pmatrix}.
		\end{align}
	\end{subequations}
	Notice that the phase function $\theta$ as well as triangular factorizations \eqref{tr} are similar as in the case of the conventional NLS equation,
	see \cite{DIZ93}.
	The sign of $\Imm\,\theta(k,\xi)$ depending on $k\in\C$ (the so-called signature table) is given in Figure \ref{stab}.
	
	\begin{figure}
		\begin{minipage}[h]{0.99\linewidth}
			\centering{\includegraphics[width=0.4\linewidth]{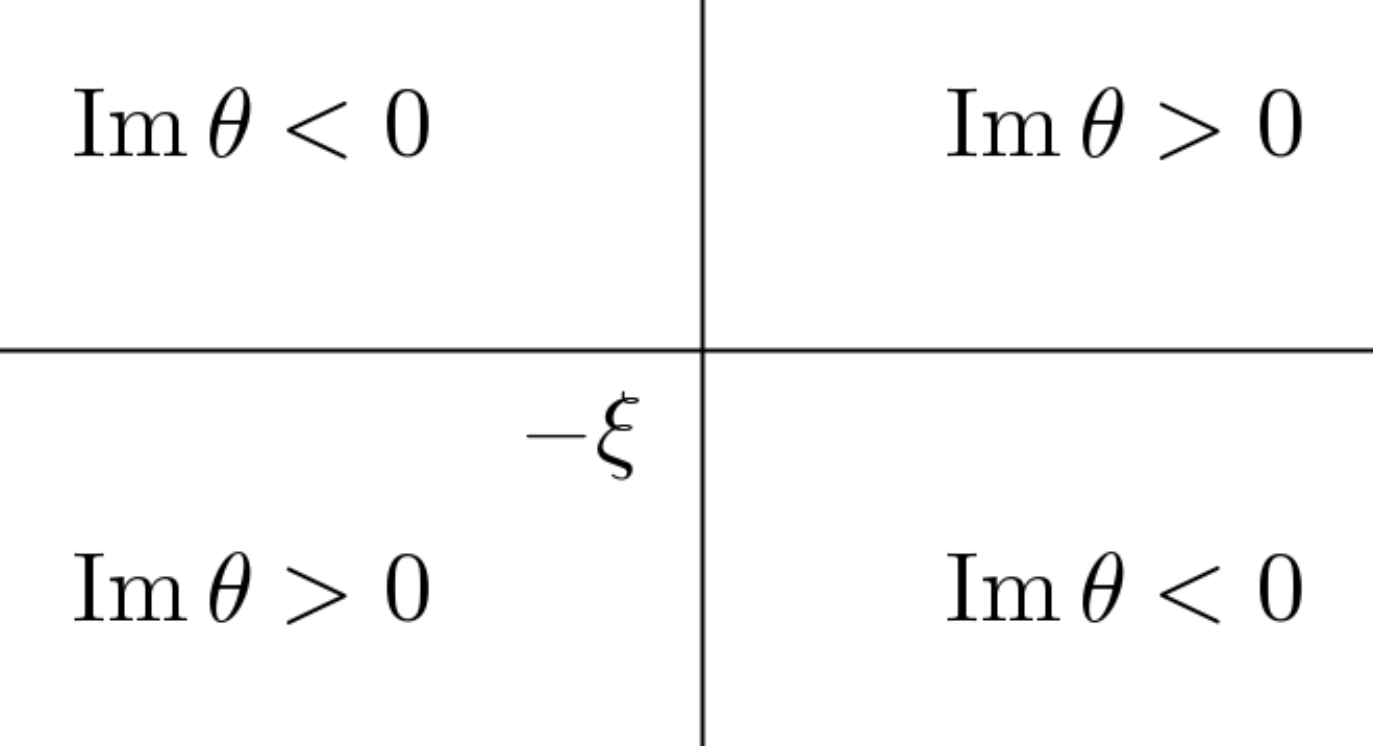}}
			\caption{Signature table of 
				$\theta(k,\xi)$.}
			\label{stab}
		\end{minipage}
	\end{figure}
	
	\subsection{Asymptotics in Case II with $n=0$}
	\label{S1}
	\subsubsection{Region $|\xi|>-B$, $B<0$}
	\label{R1}
	
	Taking into account \eqref{sol} and \eqref{sol1}, it is enough to consider 
	$-\xi<B$ only.
	Then in order to apply the ``opening lenses'' procedure, one should get rid of the diagonal factor in factorization \eqref{tr1}.
	To this end we define an auxiliary scalar sectionally analytic function $\delta(k,\xi)$, which satisfies the following RH problem:
	\begin{equation}\label{delRH}
		\begin{split}
			&\delta_+(k,\xi)=\delta_-(k,\xi)
			(1+r_1(k)r_2(k)),\quad
			k\in(-\infty,-\xi),\\
			&\delta(k,\xi)=1
			+\mathcal{O}\left(k^{-1}\right),
			\quad k\to\infty.
		\end{split}
	\end{equation}
	Scalar RH problem \eqref{delRH} can be easily reduced to the additive one, which is then solved by the using the Plemelj-Sokhotski formula.
	Thus, the solution is given in terms of the Cauchy type integral as follows:
	\begin{equation}
		\label{ddef}
		\delta(k,\xi)=\exp
		\left(
		\frac{1}{2\pi\I}
		\int_{-\infty}^{-\xi}
		\frac{\log(1+r_1(\zeta)r_2(\zeta))}
		{\zeta-k}\,d\zeta
		\right).
	\end{equation}
	Integrating by parts, we obtain the following representation for $\delta$:
	\begin{equation}\label{dsing}
		\delta(k,\xi)=
		(k+\xi)^{\I\nu(-\xi)}e^{\chi(k,\xi)},
	\end{equation}
	where
	\begin{equation}\label{chi}
		\chi(k,\xi)=-\frac{1}{2\pi\I}
		\int_{-\infty}^{-\xi}
		\log(k-\zeta)d_{\zeta}
		\log(1+ r_1(\zeta)r_2(\zeta)),
	\end{equation}
	and (see \eqref{arga} and \eqref{r-a})
	\begin{equation}\label{nu}
		\begin{split}
			\nu(-\xi)&=-\frac{1}{2\pi}
			\log(1+r_1(-\xi)r_2(-\xi))\\
			&=-\frac{1}{2\pi}\log|1+r_1(-\xi)r_2(-\xi)|
			-\frac{\I}{2\pi}
			\int_{-\infty}^{-\xi}\,
			d\arg(1+ r_1(\zeta)r_2(\zeta)).
		\end{split}
	\end{equation}
	Notice that the assumption on the winding of the argument, see II.2) in Case II, and \eqref{r-a} imply that 
	$\Imm\,\nu(-\xi)\in\left(-\frac{1}{2},\frac{1}{2}\right)$.
	
	\textbf{First transformation.}
	Using $\delta(k,\xi)$, define $\tilde{M}(x,t,k)$ as follows:
	\begin{equation*}
		\tilde{M}(x,t,k)=M(x,t,k)
		\delta^{-\sigma_3}(k,\xi).
	\end{equation*}
	Then $\tilde{M}(x,t,k)$ satisfies the following RH problem:
	\begin{description}
		\item [RH problem for 
		$\tilde{M}(x,t,k)$, $-\xi<B<0$]
		\indent
		\begin{enumerate}
			\item jump condition:
			\begin{equation}\label{til-M}
				\tilde{M}_+(x,t,k)
				=\tilde{M}_-(x,t,k)
				\tilde{J}(x,t,k),\quad
				k\in\R\setminus\{-\xi,-B\},
			\end{equation}
			with (we drop the arguments of $\tilde{J}$ here)
			\begin{equation}\label{til-J}
				\tilde{J}=
				\begin{cases}
					\begin{pmatrix}
						1& 0\\
						\frac{r_1(k)\delta_-^{-2}(k,\xi)}
						{1+r_1(k)r_2(k)}
						e^{2\I t\theta}& 1\\
					\end{pmatrix}
					\begin{pmatrix}
						1& 
						\frac{r_2(k)\delta_+^{2}(k,\xi)}
						{1+r_1(k)r_2(k)}
						e^{-2\I t\theta}\\
						0& 1\\
					\end{pmatrix},\,& k\in(-\infty,-\xi),
					\\
					\begin{pmatrix}
						1& r_2(k)\delta^2(k,\xi)
						e^{-2\I t\theta}\\
						0& 1\\
					\end{pmatrix}
					\begin{pmatrix}
						1& 0\\
						r_1(k)\delta^{-2}(k,\xi)
						e^{2\I t\theta}& 1\\
					\end{pmatrix},\,& k\in(-\xi,\infty)\setminus\{-B\};
				\end{cases}
			\end{equation}
			
			\item normalization condition at infinity:
			\begin{equation*}
				\tilde{M}(x,t,k)=I
				+\mathcal{O}\left(k^{-1}\right)
				,\quad k\to\infty;
			\end{equation*}
			
			\item singularities at $k=-\xi$ and
			$k=-B$ (see \eqref{nu}):
			\begin{equation*}
				\begin{split}
					&\tilde{M}(x,t,k)
					\begin{pmatrix}
						(k+\xi)^{-\Imm\,\nu(-\xi)}& 0\\
						0& (k+\xi)^{\Imm\,\nu(-\xi)}
					\end{pmatrix}
					=\mathcal{O}(1),\quad
					k\to -\xi,
				\end{split}
			\end{equation*}
			and
			\begin{equation}
				\label{til-M--B}
				\begin{split}
					&\tilde{M}(x,t,k)
					\begin{pmatrix}
						(k+B)^{-1}& 0\\
						0& k+B
					\end{pmatrix}
					=\mathcal{O}(1),\quad
					k\to-B+\I 0,\\
					&\tilde{M}(x,t,k)
					=\mathcal{O}(1),\quad
					k\to-B-\I 0.
				\end{split}
			\end{equation}
		\end{enumerate}
	\end{description}
	
	\textbf{Second transformation.}
	Now we are at the position to transform the RH problem for $\tilde{M}$ in such a way that the jump matrix of the transformed problem converges, as $t\to\infty$, to the identity matrix.
	Let us define $\hat{M}$ as follows 
	(see Figure \ref{F-1}):
	\begin{equation}\label{hat-M}
		\hat{M}(x,t,k)=
		\begin{cases}
			\tilde{M}(x,t,k),
			& k\in\hat\Omega_0,\\
			\tilde{M}(x,t,k)
			\begin{pmatrix}
				1& \frac{-r_2(k)\delta^{2}(k,\xi)}
				{1+r_1(k)r_2(k)}
				e^{-2\I t\theta}\\
				0& 1\\
			\end{pmatrix},
			& k\in\hat\Omega_1,\\
			\tilde{M}(x,t,k)
			\begin{pmatrix}
				1& 0\\
				-r_1(k)\delta^{-2}(k,\xi)
				e^{2\I t\theta}& 1\\
			\end{pmatrix},
			& k\in\hat\Omega_2,\\
			\tilde{M}(x,t,k)
			\begin{pmatrix}
				1& r_2(k)\delta^2(k,\xi)
				e^{-2\I t\theta}\\
				0& 1\\
			\end{pmatrix},
			& k\in\hat\Omega_3,\\
			\tilde{M}(x,t,k)
			\begin{pmatrix}
				1& 0\\
				\frac{r_1(k)\delta^{-2}(k,\xi)}
				{1+r_1(k)r_2(k)}
				e^{2\I t\theta}& 1\\
			\end{pmatrix},
			& k\in\hat\Omega_4.
		\end{cases}
	\end{equation}
	\begin{figure}
		\begin{minipage}[h]{0.99\linewidth}
			\centering{\includegraphics[width=0.7\linewidth]{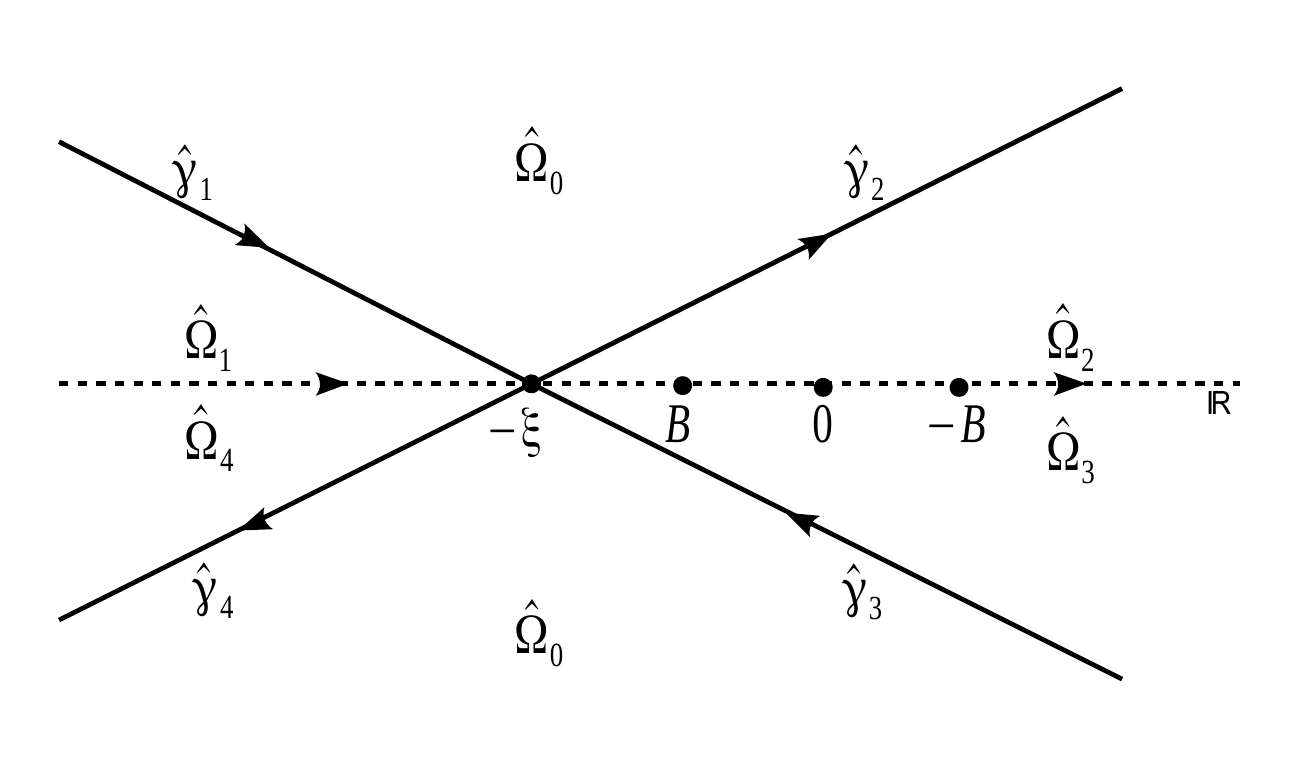}}
			\caption{Contour $\hat\Gamma
				=\hat\gamma_1\cup\dots
				\cup\hat\gamma_4$
				and domains $\hat\Omega_j$,
				$j=0,\dots,4$ in Section \ref{R1}.}
			%for $-\xi<B$, $B<0$.}
		\label{F-1}
	\end{minipage}
\end{figure}
Then using \eqref{til-M}, \eqref{til-J} and \eqref{hat-M}, we obtain that $\hat{M}(x,t,k)$ satisfies the following  jump condition on the cross 
$\hat\Gamma=\bigcup\limits_{j=1}^4
\hat\gamma_j$ (see Figure \ref{F-1}):
\begin{equation}\label{hat-M-j}
	\hat{M}_+(x,t,k)
	=\hat{M}_-(x,t,k)
	\hat{J}(x,t,k),\quad
	k\in\hat\Gamma,
\end{equation}
with the jump matrix $\hat{J}(x,t,k)$ given by (we drop the arguments of $\hat{J}$ here)
\begin{equation}
	\label{hat-J}
	\hat{J}=
	\begin{cases}
		\begin{pmatrix}
			1& \frac{r_2(k)\delta^{2}(k,\xi)}
			{1+r_1(k)r_2(k)}
			e^{-2\I t\theta}\\
			0& 1\\
		\end{pmatrix},
		& k\in\hat\gamma_1,\\
		\begin{pmatrix}
			1& 0\\
			r_1(k)\delta^{-2}(k,\xi)
			e^{2\I t\theta}& 1\\
		\end{pmatrix},
		& k\in\hat\gamma_2,\\
		\begin{pmatrix}
			1& -r_2(k)\delta^2(k,\xi)
			e^{-2\I t\theta}\\
			0& 1\\
		\end{pmatrix},
		& k\in\hat\gamma_3,\\
		\begin{pmatrix}
			1& 0\\
			\frac{-r_1(k)\delta^{-2}(k,\xi)}
			{1+r_1(k)r_2(k)}
			e^{2\I t\theta}& 1\\
		\end{pmatrix},
		& k\in\hat\gamma_4.
	\end{cases}
\end{equation}

Let us study the behavior of $\hat{M}$ in a neighborhood of the singular point $k=-B$.
From \eqref{hat-M}, \eqref{til-M--B} and \eqref{rj-s}
we have
\begin{equation}\label{hat-M-s}
	\begin{split}
		&\hat{M}(x,t,k)
		\begin{pmatrix}
			(k+B)^{-1}&0\\
			\frac{b(-B)}{a_1^{-B}}
			\delta^{-2}(-B,\xi)
			e^{2\I t\theta(-B,\xi)}& k+B
		\end{pmatrix}
		=\mathcal{O}(1),\quad
		k\to-B+\I0,\\
		&\hat{M}(x,t,k)
		\begin{pmatrix}
			1
			&\frac{\overline{b^B}}{a_2(B)}
			\delta^2(-B,\xi)
			e^{-2\I t\theta(-B,\xi)}
			(k+B)^{-1}\\
			0& 1
		\end{pmatrix}
		=\mathcal{O}(1),\quad
		k\to-B-\I0.
	\end{split}
\end{equation}
Using \eqref{scoeff} and recalling that 
$a_2(-B)=\overline{a_2}(B)$, we obtain
\begin{equation}\label{sc-res}
	\frac{a_1^{-B}}{b(-B)}
	=\frac{\overline{b^B}}{a_2(-B)}
	=-\frac{A}{2\I}.
\end{equation}
Combining \eqref{hat-M-s} and \eqref{sc-res}, we conclude that $\hat{M}^{(2)}$ satisfies the following residue condition at $k=-B$:
\begin{equation}\label{res--B}
	\underset{k=-B}
	{\operatorname{Res}}
	\hat{M}^{(2)}(x,t,k)=
	\frac{A}{2\I}
	\delta^2(-B,\xi)
	e^{2\I Bx-4\I B^2t}
	\hat{M}^{(1)}(x,t,-B).
\end{equation}
Thus we arrive at the following
\begin{description}
	\item [RH problem for $\hat{M}(x,t,k)$, $-\xi<B<0$]
	\indent
	\begin{enumerate}
		\item jump condition \eqref{hat-M-j}--\eqref{hat-J} on the cross $\hat\Gamma$ given in Figure \ref{F-1}; 
		
		\item normalization condition at infinity:
		\begin{equation*}
			\hat{M}(x,t,k)=I
			+\mathcal{O}\left(k^{-1}\right)
			,\quad k\to\infty;
		\end{equation*}
		
		\item residue condition \eqref{res--B} at $k=-B$;
		\item singularity at $k=-\xi$ (see \eqref{nu}):
		\begin{equation*}
			\begin{split}
				&\hat{M}(x,t,k)
				\begin{pmatrix}
					(k+\xi)^{-\Imm\,\nu(-\xi)}& 0\\
					0& (k+\xi)^{\Imm\,\nu(-\xi)}
				\end{pmatrix}
				=\mathcal{O}(1),\quad
				k\to -\xi.
			\end{split}
		\end{equation*}
	\end{enumerate}
\end{description}

Recalling that $\Imm\,\nu(-\xi)\in\left(-\frac{1}{2},\frac{1}{2}\right)$, we conclude that the RH problem on the cross for 
$\hat{M}$ vanishes, as per \cite{RS19}. 
Then using inverse transform \eqref{sol} and \eqref{sol1}, we obtain the rough asymptotics for the solution $q(x,t)$ presented in Proposition \ref{RA1}, item 2), for $|\xi|>-B$.

\subsubsection{Region $0\leq|\xi|<-B$, $B<0$}
\label{R2}
For such values of $\xi$ the main challenge is to get rid of the diagonal factor in \eqref{tr1}.
Here the jump matrix of the RH problem for the auxiliary scalar function $\delta$ has a zero on the contour at $k=B$, see \eqref{delRH} and \eqref{r-a}.
Therefore one cannot directly apply the Plemelj-Sokhotski formula to this problem (cf.\,\,\eqref{ddef}) and should take additional care of the singularity at $k=B$.

Take arbitrary $\tilde{k}_0\in\C^+$.
Define a scalar function $\delta_1(k,\xi;B)$ as a solution of the following RH problem:
\begin{equation}\label{del1RH}
	\begin{split}
		&\delta_{1+}(k,\xi;B)
		=\delta_{1-}(k,\xi;B)
		\frac{k+\tilde{k}_0}{k-B}
		(1+r_1(k)r_2(k)),\quad
		k\in(-\infty,-\xi)\setminus\{B\},\\
		&\delta_1(k,\xi;B)=1
		+\mathcal{O}\left(k^{-1}\right),
		\quad k\to\infty.
	\end{split}
\end{equation}
Since the jump matrix in \eqref{del1RH} is nonzero at $k=B$ (see \eqref{r-a} and Assumption A.1)--A.2)) and it converges to $1$ as $k\to-\infty$, we have, by the Plemelj-Sokhotski formula, that
\begin{equation}
	\label{d1def}
	\begin{split}
		&\delta_1(k,\xi;B)=\exp
		\left(
		\frac{1}{2\pi\I}
		\int_{-\infty}^{-\xi}
		\frac{\log\left(
			\frac{\zeta+\tilde{k}_0}{\zeta-B}
			(1+r_1(\zeta)r_2(\zeta))
			\right)}
		{\zeta-k}\,d\zeta
		\right),\\
		&\delta_1(k,\xi;B)
		=\mathcal{O}(1),\quad k\to B.
	\end{split}
\end{equation}
Then define a scalar function $\delta_2$ as a solution of the RH problem on the contour $(-\xi,\infty)$, which reads:
\begin{equation}\label{del2RH}
	\begin{split}
		&\delta_{2+}(k,\xi;B)
		=\delta_{2-}(k,\xi;B)
		\frac{k+\tilde{k}_0}{k-B},\quad
		k\in(-\xi,\infty),\\
		&\delta_2(k,\xi;B)=1
		+\mathcal{O}\left(k^{-1}\right),
		\quad k\to\infty.
	\end{split}
\end{equation}
A unique solution of \eqref{del2RH} is given by
\begin{equation}
	\label{d2def}
	\delta_2(k,\xi;B)=\exp
	\left(
	\frac{1}{2\pi\I}
	\int_{-\xi}^{\infty}
	\frac{\log\left(
		\frac{\zeta+\tilde{k}_0}{\zeta-B}
		\right)}
	{\zeta-k}\,d\zeta
	\right).
\end{equation}
Finally, using $\delta_1$ and $\delta_2$, we define $\hat\delta(k,\xi;B)$ as follows:
\begin{equation}
	\label{hat-del}
	\hat\delta(k,\xi;B)=
	\begin{cases}
		\frac{k-B}{k+\tilde{k}_0}
		\delta_1(k,\xi;B)
		\delta_2(k,\xi;B),
		&k\in\C^+,\\
		\delta_1(k,\xi;B)
		\delta_2(k,\xi;B),
		&k\in\C^-.
	\end{cases}
\end{equation}

Integrating by parts in \eqref{d1def} and \eqref{d2def}, we conclude that the product $\delta_1\delta_2$ has the following behavior at $k=-\xi$
(cf.\,\,\eqref{dsing})
\begin{equation}\label{d1d2s}
	\delta_1(k,\xi;B)\delta_2(k,\xi;B)=
	(k+\xi)^{\I\nu(-\xi)}
	e^{\chi_1(k,\xi;B)+\chi_2(k,\xi;B)},
\end{equation}
where $\nu(-\xi)$ has the same expression as in \eqref{nu} and (cf.\,\,\eqref{chi})
\begin{equation}
	\label{chi1,2}
	\begin{split}
		&\chi_1(k,\xi;B)=-\frac{1}{2\pi\I}
		\int_{-\infty}^{-\xi}
		\log(k-\zeta)d_{\zeta}
		\log\left(
		\frac{\zeta+\tilde{k}_0}{\zeta-B}
		(1+ r_1(\zeta)r_2(\zeta))\right),\\
		&\chi_2(k,\xi;B)=-\frac{1}{2\pi\I}
		\int_{-\xi}^{\infty}
		\log(k-\zeta)d_{\zeta}
		\log\left(
		\frac{\zeta+\tilde{k}_0}{\zeta-B}
		\right).
	\end{split}
\end{equation}
Therefore $\hat\delta$ can be written in the form (see \eqref{nu} and \eqref{chi1,2})
\begin{equation}
	\label{hat-del-s}
	\hat\delta(k,\xi;B)=
	\begin{cases}
		(k+\xi)^{\I\nu(-\xi)}
		\frac{k-B}{k+\tilde{k}_0}
		e^{\chi_1(k,\xi;B)+\chi_2(k,\xi;B)},
		&k\in\C^+,\\
		(k+\xi)^{\I\nu(-\xi)}
		e^{\chi_1(k,\xi;B)+\chi_2(k,\xi;B)},
		&k\in\C^-.
	\end{cases}
\end{equation}

Combining \eqref{del1RH}, \eqref{del2RH} and \eqref{hat-del}, we conclude that $\hat{\delta}$ satisfies the following scalar RH problem:
\begin{description}
	\item [RH problem for 
	$\hat{\delta}(k,\xi;B)$]
	\indent
	\begin{enumerate}
		\item jump condition:
		\begin{equation}\label{hat-d-j}
			\hat{\delta}_+(k,\xi;B)
			=\hat{\delta}_-(k,\xi;B)
			(1+r_1(k)r_2(k)),\quad
			k\in(-\infty,-\xi)
			\setminus\{B\};
		\end{equation}
		\item normalization condition:
		\begin{equation}\label{hat-d-n}
			\hat\delta(k,\xi;B)
			=1+\mathcal{O}(1),
			\quad k\to\infty;
		\end{equation}
		\item behavior at $k=B$:
		\begin{equation*}
			\begin{split}
				&\hat\delta(k,\xi;B)=
				(k-B)\frac{\delta_{1+}(B,\xi;B)
					\delta_2(B,\xi;B)}
				{B+\tilde{k}_0}
				+\mathcal{O}\left((k-B)^2\right),
				\quad k\to B+\I0,\\
				&\hat\delta(k,\xi;B)=
				\delta_{1-}(B,\xi;B)
				\delta_2(B,\xi;B)
				+\mathcal{O}(k-B),
				\quad k\to B-\I0;
			\end{split}
		\end{equation*}
		\item singularity at $k=-\xi$ is described in \eqref{hat-del-s}.
	\end{enumerate}
\end{description}

\begin{figure}
	\begin{minipage}[h]{0.99\linewidth}
		\centering{\includegraphics[width=0.7\linewidth]{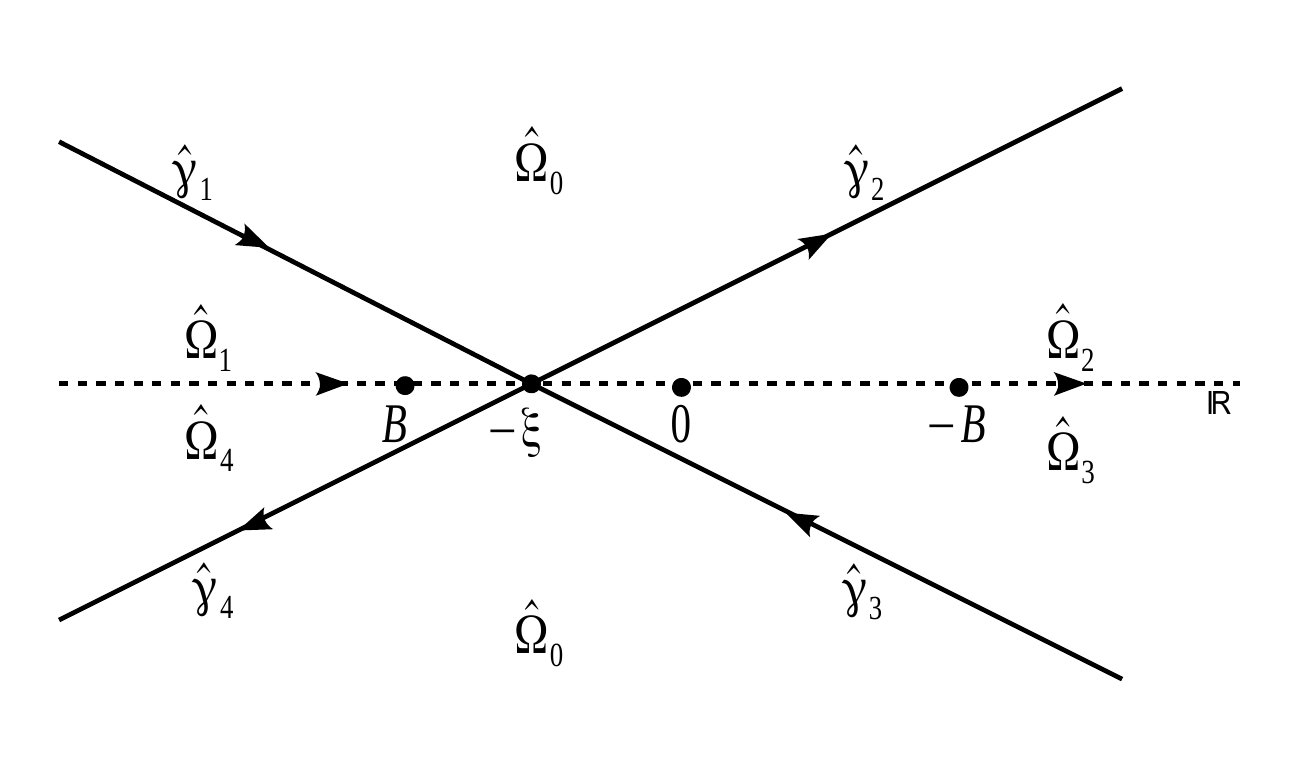}}
		\caption{Contour $\hat\Gamma
			=\hat\gamma_1\cup\dots
			\cup\hat\gamma_4$
			and domains $\hat\Omega_j$,
			$j=0,\dots,4$ in Section \ref{R2}.} 
		%for $0\leq|\xi|<-B$, $B<0$.}
	\label{F-2}
\end{minipage}
\end{figure}

Now we make the same transformations 
$M\rightsquigarrow\tilde{M}
\rightsquigarrow\hat{M}$
as in Section \ref{R1}, but with 
$\hat{\delta}(k,\xi;B)$ instead of $\delta(k,\xi)$ and with
domains $\hat\Omega_j$, $j=0,\dots,4$,
and contour $\hat\Gamma$ given in Figure \ref{F-2}.
Direct calculations show that 
$\hat{M}(x,t,k)$ has the following behavior at $k=B$ (we drop the arguments of $\hat{M}(x,t,k)$ and write $\delta_{1\pm}(B,\xi)=\delta_{1\pm}(B,\xi;B)$,
$\delta_2(B,\xi)=\delta_2(B,\xi;B)$ here):
\begin{equation}\label{h-M-s}
\begin{split}
	&\hat{M}
	\begin{pmatrix}
		\frac{k-B}{k+\tilde{k}_0}
		(\delta_{1+}\delta_2)(B,\xi)
		&\frac{a_1^B}
		{k+\tilde{k}_0}
		(r_2a_2)(B)
		(\delta_{1+}\delta_2)(B,\xi)
		e^{-2\I t\theta(B,\xi)}\\
		0
		&\frac{k+\tilde{k}_0}{k-B}
		(\delta_{1+}
		\delta_2)^{-1}(B,\xi)
	\end{pmatrix}
	=\mathcal{O}(1),\quad
	k\to B+\I0,\\
	&\hat{M}
	\begin{pmatrix}
		(\delta_{1-}\delta_2)(B,\xi)
		&0\\
		-\frac{a_1^{B}}{k-B}
		(r_1a_2)(B)
		(\delta_{1-}
		\delta_2)^{-1}(B,\xi)
		e^{2\I t\theta(B,\xi)}
		&(\delta_{1-}
		\delta_2)^{-1}(B,\xi)
	\end{pmatrix}
	=\mathcal{O}(1),\quad
	k\to B-\I0.
\end{split}
\end{equation}

Expansions \eqref{h-M-s} imply that the residue condition of $\hat M^{(1)}$ can be written as follows:
\begin{equation}\label{res-B1}
\underset{k=B}
{\operatorname{Res}}
\hat{M}^{(1)}(x,t,k)=
-(B+\tilde{k}_0)^2
\frac{
	\left(
	\delta_{1+}\delta_2
	\right)^{-2}(B,\xi;B)}
{a_1^B(r_2a_2)(B)}
e^{2\I Bx+4\I B^2t}
\hat{M}^{(2)}(x,t,B),
\end{equation}
or
\begin{equation}\label{res-B2}
\underset{k=B}
{\operatorname{Res}}
\hat{M}^{(1)}(x,t,k)=
a_1^B(r_1a_2)(B)
\left(
\delta_{1-}\delta_2
\right)^{-2}(B,\xi;B)
e^{2\I Bx+4\I B^2t}
\hat{M}^{(2)}(x,t,B).
\end{equation}
Taking into account the jump condition of $\delta_1$ at $k=B$ (see \eqref{del1RH},
\eqref{r-a} and \eqref{a1pmB})
$$
\delta_{1+}(B,\xi;B)
=\delta_{1-}(B,\xi;B)
\frac{B+\tilde{k}_0}{a_1^Ba_2(B)},
$$
and using \eqref{scoeff}, \eqref{rj-s}, we conclude that (see \eqref{res-B1}; for simplicity we drop the third argument of $\delta_{1\pm}(B,\xi;B)$ and
$\delta_2(B,\xi;B)$ here)
\begin{equation*}
-(B+\tilde{k}_0)^2
\frac{
	\left(
	\delta_{1+}\delta_2
	\right)^{-2}(B,\xi)}
{a_1^B(r_2a_2)(B)}
=-a_1^Ba_2(B)
\frac{
	\left(
	\delta_{1-}\delta_2
	\right)^{-2}(B,\xi)}
{r_2(B)}
=\frac{A}{2\I}a_2^2(B)
\left(
\delta_{1-}\delta_2
\right)^{-2}(B,\xi).
\end{equation*}
Using \eqref{scoeff} and \eqref{rj-s} in \eqref{res-B2} we obtain
\begin{equation*}
a_1^B(r_1a_2)(B)
=b^Ba_2(B)
=\frac{A}{2\I}a_2^2(B),
\end{equation*}
which implies that \eqref{res-B1} and \eqref{res-B2} are consistent.
Thus, eventually we arrive at the following residue condition at $k=B$:
\begin{equation}\label{res-B}
\underset{k=B}
{\operatorname{Res}}
\hat{M}^{(1)}(x,t,k)=
\frac{A}{2\I}a_2^2(B)
\hat\delta_-^{-2}
(B,\xi;B)
e^{2\I Bx+4\I B^2t}
\hat{M}^{(2)}(x,t,B).
\end{equation}
The Riemann-Hilbert problem for $\hat{M}$ reads as follows:
\begin{description}
\item [RH problem for $\hat{M}(x,t,k)$, $0\leq|\xi|<-B$]
\indent
\begin{enumerate}
	\item jump condition \eqref{hat-M-j}--\eqref{hat-J} with $\hat\delta(k,\xi;B)$ (see \eqref{hat-del})
	instead of $\delta(k,\xi)$ on the cross $\hat\Gamma$ given in Figure \ref{F-2}; 
	
	\item normalization condition at infinity:
	\begin{equation*}
		\hat{M}(x,t,k)=I
		+\mathcal{O}\left(k^{-1}\right)
		,\quad k\to\infty;
	\end{equation*}
	
	\item residue condition
	\eqref{res-B} at $k=B$;
	
	\item residue condition \eqref{res--B} at $k=-B$ with $\hat\delta(k,\xi;B)$
	instead of $\delta(k,\xi)$;
	
	\item singularity at $k=-\xi$ (see \eqref{nu}):
	\begin{equation*}
		\begin{split}
			&\hat{M}(x,t,k)
			\begin{pmatrix}
				(k+\xi)^{-\Imm\,\nu(-\xi)}& 0\\
				0& (k+\xi)^{\Imm\,\nu(-\xi)}
			\end{pmatrix}
			=\mathcal{O}(1),\quad
			k\to -\xi.
		\end{split}
	\end{equation*}
\end{enumerate}
\end{description}

Now let us calculate the rough asymptotics of $q(x,t)$ as 
$t\to\infty$.
\begin{proposition}[Rough asymptotics, $0\leq|\xi|<-B$]
In the Case II with $n=0$ and $B<0$ we have the long-time asymptotic behavior of $q(x,t)$ given in Proposition \ref{RA1}, item 2) for
$0\leq|\xi|<-B$:
\begin{equation}
	\label{ras1}
	q(x,t)=
	\frac{16AB^2\hat{\delta}^2(-B,\xi;B)
		e^{2\I Bx-4\I B^2t}}
	{16B^2
		-A^2\left(\overline{\hat{\delta}}\right)^2
		(-B,-\xi;B)
		\hat{\delta}^2(-B,\xi;B)
		e^{4\I Bx}}
	+\so(1),\quad t\to\infty.
\end{equation}
Here the asymptotic formula is valid away from the arbitrary small neighborhoods of possible zeros of the denominator, and $\hat\delta(k,\xi;B)$ is given in \eqref{hat-del} (see also \eqref{hat-d-j}).
\end{proposition}
\begin{proof}
Since the jump matrix in the RH problem for $\hat{M}$ vanishes as $t\to\infty$, we have
\begin{equation*}
	\hat{M}(x,t,k)
	=\hat{M}^{\mathrm{as}}
	(x,t,k)
	+\so(1),\quad t\to\infty,
\end{equation*}
where 
$\hat{M}^{\mathrm{as}}$
is a meromorphic $2\times2$ matrix which satisfies the following conditions:
\begin{equation}
	\label{RHMas}
	\begin{split}
		&\underset{k=B}
		{\operatorname{Res}}
		\left(
		\hat{M}^{\mathrm{as}}
		\right)^{(1)}(x,t,k)=
		c_1(x,t)
		\left(
		\hat{M}^{\mathrm{as}}
		\right)^{(2)}(x,t,B),\\
		&\underset{k=-B}
		{\operatorname{Res}}
		\left(
		\hat{M}^{\mathrm{as}}
		\right)^{(2)}(x,t,k)=
		c_2(x,t)
		\left(
		\hat{M}^{\mathrm{as}}
		\right)^{(1)}(x,t,-B),\\
		&\hat{M}^{\mathrm{as}}
		(x,t,k)=
		I+\mathcal{O}
		\left(k^{-1}\right)
		,\quad k\to\infty,
	\end{split}
\end{equation}
where (see \eqref{res-B} and
\eqref{res--B} with $\hat\delta$ instead of $\delta$)
\begin{equation}\label{cj}
	\begin{split}
		&c_1(x,t)=
		\frac{A}{2\I}a_2^2(B)
		\hat\delta_-^{-2}
		(B,\xi;B)
		e^{2\I Bx+4\I B^2t},\\
		&c_2(x,t)=
		\frac{A}{2\I}
		\hat{\delta}^2(-B,\xi;B)
		e^{2\I Bx-4\I B^2t}.
	\end{split}
\end{equation}
Let us seek the solution of \eqref{RHMas} in the form
\begin{equation*}
	\hat{M}^{\mathrm{as}}
	(x,t,k)=
	\begin{pmatrix}
		\frac{k+A_1}{k-B}&
		\frac{A_3}{k+B}\\
		\frac{A_2}{k-B}&
		\frac{k+A_4}{k+B}
	\end{pmatrix},
\end{equation*}
with some $A_j=A_j(x,t)$,
$j=1,\dots,4$.
From the residue conditions in \eqref{RHMas} we conclude that
\begin{equation}\label{Aj}
	\begin{split}
		&A_1(x,t)=\frac
		{B(c_1c_2)(x,t)-4B^3}
		{4B^2+(c_1c_2)(x,t)},\quad
		A_2(x,t)=\frac
		{4B^2c_1(x,t)}
		{4B^2+(c_1c_2)(x,t)},\\
		&A_3(x,t)=\frac
		{4B^2c_2(x,t)}
		{4B^2+(c_1c_2)(x,t)},\quad
		A_4(x,t)=\frac
		{4B^3-B(c_1c_2)(x,t)}
		{4B^2+(c_1c_2)(x,t)}.
	\end{split}
\end{equation}
We conclude from \eqref{sol} that
%(cf.\,\,\eqref{sol}--\eqref{sol1})
\begin{equation}\label{qas}
	\begin{split}
		&q(x,t)=2\I\lim_{k\to\infty}
		k\hat{M}_{1,2}^{\mathrm{as}}
		(x,t,k)+\so(1)
		=2\I A_3(x,t)+\so(1),%\\
		%&q(-x,t)=-2\I\lim_{k\to\infty}
		%k\overline{\hat{M}_{2,1}
			%^{\mathrm{as}}}
		%(x,t,k)+\so(1),
	\end{split}
\end{equation}
as $t\to\infty$, for $0\leq|\xi|<-B$.
To simplify the expression for $A_3$ given in \eqref{Aj}, let us show that (in what follows we will drop the third argument of 
$\hat\delta(k,\xi;B)$ and $\delta_j(k,\xi;B)$, $j=1,2$)
\begin{equation}\label{a_2(B)}
	a_2(B)=\hat\delta_-(B,\xi)
	\overline{\hat\delta}(-B,-\xi),\quad
	0\leq|\xi|<-B.
\end{equation}
Indeed, using the Plemelj-Sokhotski formula, we obtain from \eqref{d1def}, \eqref{d2def} and \eqref{hat-del} (recall \eqref{r-a})
\begin{equation}\label{hdBxi}
	\begin{split}
		\hat\delta_-(B,\xi)&=\hat\delta_{1-}(B,\xi)
		\hat\delta_2(B,\xi)\\
		&=\exp\left(
		-\frac{1}{2}\log\left(\frac{B+\tilde{k}_0}
		{a_1^Ba_2(B)}\right)
		+\frac{1}{2\pi\I}\mathrm{v.p.}
		\int_{-\infty}^{-\xi}
		\frac{\log\left(
			\frac{\zeta+\tilde{k}_0}{\zeta-B}
			(1+r_1(\zeta)r_2(\zeta))
			\right)}
		{\zeta-B}\,d\zeta
		\right)\\
		&\quad\times
		\exp\left(
		\frac{1}{2\pi\I}
		\int_{-\xi}^{\infty}
		\frac{\log\left(
			\frac{\zeta+\tilde{k}_0}{\zeta-B}
			\right)}
		{\zeta-B}\,d\zeta
		\right),
	\end{split}
\end{equation}
and
\begin{equation}\label{hd-B-xi}
	\begin{split}
		\overline{\hat\delta}(-B,-\xi)
		=&\exp\left(
		\frac{-1}{2\pi\I}
		\int_{-\infty}^{\xi}
		\frac{\log\left(
			\frac{\zeta+\overline{\tilde{k}_0}}
			{\zeta-B}
			(1+\overline{r_1}(\zeta)\overline{r_2}(\zeta))
			\right)}
		{\zeta+B}\,d\zeta\right)\\
		&\times\exp\left(
		-\frac{1}{2}\log\left(
		\frac{B-\overline{\tilde{k}_0}}{2B}\right)
		-\frac{1}{2\pi\I}
		\mathrm{v.p.}\int_\xi^\infty
		\frac{\log\left(
			\frac{\zeta+\overline{\tilde{k}_0}}
			{\zeta-B}
			\right)}
		{\zeta+B}\,d\zeta
		\right).
	\end{split}
\end{equation}
Using symmetry relation \eqref{r-sym}, we conclude from \eqref{hdBxi} and \eqref{hd-B-xi} that
\begin{equation}\label{delpr}
	\begin{split}
		\hat\delta_-(B,\xi)
		\overline{\hat\delta}(-B,-\xi)=
		\exp&\left(
		-\frac{1}{2}\log\left(
		\frac{\left(B+\tilde{k}_0\right)
			\left(B-\overline{\tilde{k}_0}\right)}
		{2Ba_1^Ba_2(B)}\right)\right.\\
		&\left.\quad+\frac{1}{2\pi\I}
		\mathrm{v.p.}\int_{-\infty}^{\infty}
		\frac{\log\left(
			\frac{\left(\zeta+\tilde{k}_0\right)
				\left(\zeta-\overline{\tilde{k}_0}\right)}
			{(\zeta-B)(\zeta+B)}
			(1+r_1(\zeta)r_2(\zeta))
			\right)}
		{\zeta-B}\,d\zeta
		\right).
	\end{split}
\end{equation}
Now we rewrite \eqref{r-a} as a Riemann-Hilbert problem for $a_1^{-1}(k)$ and $a_2(k)$ as follows:
\begin{equation*}
	a_1^{-1}(k)=a_2(k)(1+r_1(k)r_2(k)),\quad
	k\in\mathbb{R}.
\end{equation*}
Using the Plemelj-Sokhotski formula, we obtain the following expressions for $a_j(k)$, $j=1,2$ (recall that $\tilde{k}_0\in\C^+$):
\begin{equation*}
	\begin{split}
		a_1(k)=&\frac{\left(k+\tilde{k}_0\right)
			\left(k-\overline{\tilde{k}_0}\right)}
		{(k-B)(k+B)}\\
		&\times\exp\left(
		\frac{1}{2\pi\I}\int_{-\infty}^{\infty}
		\frac{\log\left(
			\frac{\left(\zeta+\tilde{k}_0\right)
				\left(\zeta-\overline{\tilde{k}_0}\right)}
			{(\zeta-B)(\zeta+B)}
			(1+r_1(\zeta)r_2(\zeta))
			\right)}
		{\zeta-k}\,d\zeta\right),\quad k\in\C^+,
	\end{split}
\end{equation*}
and
\begin{equation}\label{a-2exp}
	a_2(k)=\exp\left(
	\frac{1}{2\pi\I}\int_{-\infty}^{\infty}
	\frac{\log\left(
		\frac{\left(\zeta+\tilde{k}_0\right)
			\left(\zeta-\overline{\tilde{k}_0}\right)}
		{(\zeta-B)(\zeta+B)}
		(1+r_1(\zeta)r_2(\zeta))
		\right)}
	{\zeta-k}\,d\zeta\right),\quad k\in\C^-.
\end{equation}
Combining \eqref{delpr}, \eqref{a-2exp} and \eqref{r-a}, we arrive at \eqref{a_2(B)}.
Finally, equations \eqref{qas}, \eqref{Aj} and \eqref{a_2(B)} yield \eqref{ras1}.
\end{proof}
\begin{remark}
The asymptotics of the solution $q(x,t)$ in the region $0\leq|\xi|<-B$ can be found also from \eqref{sol1} in terms of $A_2(x,t)$ as follows (cf.\,\,\eqref{qas})
\begin{equation}\label{qas1}
	q(x,t)=-2\I\lim_{k\to\infty}
	k\overline{\hat{M}_{2,1}^{\mathrm{as}}}(-x,t,k)
	+\so(1)
	=-2\I\overline{A_2}(-x,t)+\so(1).
\end{equation}
Let us show that \eqref{qas} and \eqref{qas1} are consistent.
Observe that the equality $A_3(x,t)=-\overline{A_2}(-x,t)$ is equivalent to the following equation
(see \eqref{Aj} and \eqref{cj})
\begin{equation*}
	4B^2(\overline{c_1}(-x,t)+c_2(x,t))
	+\overline{c_1}(-x,t)c_2(x,t)
	(c_1(x,t)+\overline{c_2}(-x,t))=0.
\end{equation*}
Equality
$c_1(x,t)+\overline{c_2}(-x,t)=0$ follows from \eqref{a_2(B)} and thus we have that $A_3(x,t)=\overline{A_2}(-x,t)$ for all $\xi$ such that $0\leq|\xi|<-B$.
\end{remark}
\begin{remark}[Independence of \eqref{ras1} of $\tilde{k}_0$]
In the definition of $\hat\delta(k,\xi)$ we have used an auxiliary point $\tilde{k}_0\in\C^+$, see \eqref{hat-del}.
Taking into account that $\hat\delta$ satisfies the same RH problem \eqref{hat-d-j}--\eqref{hat-d-n} for any $\tilde{k}_0\in\C^+$, we conclude that it does not depend on the choice of $\tilde{k}_0$.
\end{remark}

%\subsection{Case II with $n=0$ and $B>0$.}
%\label{S2}
\subsubsection{Region $|\xi|>B$, $B>0$}
\label{S2R1}
In this region the asymptotic analysis of the basic RH problem follows the similar lines as in the case $B<0$, see Section \ref{R1}.
Thus, we obtain the same asymptotic behavior of $q(x,t)$ as described in Proposition \ref{RA1}, but with $B$ instead of $-B$:
%\begin{proposition}[Rough asymptotics, $|\xi|>B$]
%	\label{RA2}
%	In the Case II with $n=0$ and $B<0$ we have the following long-time asymptotic behavior of $q(x,t)$ for 
%	$|\xi|>B$, $\xi=\frac{x}{4t}$:
%	\begin{equation*}
%		q(x,t)=
%		\begin{cases}
	%			\so(1),&\xi<-B,\\
	%			A\delta^2(-B,\xi)
	%			e^{2\I Bx-4\I B^2t}
	%			+\so(1),&\xi>B,
	%		\end{cases}
%	\end{equation*}
%	as $t\to\infty$.
%	Here $\delta(-B,\xi)$ is given in \eqref{ddef} with $k=-B$.
%\end{proposition}

\subsubsection{Region $0\leq|\xi|<B$, $B>0$}
\label{S2R2}
As in Section \ref{R2}, the jump matrix of the auxiliary scalar RH problem for $\delta$ has a simple zero on the contour at $k=-B$, see \eqref{delRH} and \eqref{r-a}.
Substituting $-B$ instead of $B$ in \eqref{del1RH}--\eqref{hat-del-s}, we obtain the auxiliary function $\hat\delta(k,\xi;-B)$ defined in \eqref{hat-del} (see also \eqref{hat-del-s}), which solves the following scalar RH problem (cf.\,\,Section \ref{R2}):
\begin{description}
\item [RH problem for 
$\hat{\delta}(k,\xi;-B)$]
\indent
\begin{enumerate}
	\item jump condition:
	\begin{equation}\label{hat-d-j-B}
		\hat{\delta}_+(k,\xi;-B)
		=\hat{\delta}_-(k,\xi;-B)
		(1+r_1(k)r_2(k)),\quad
		k\in(-\infty,-\xi)
		\setminus\{-B\};
	\end{equation}
	\item normalization condition:
	\begin{equation}\label{hat-d-n-B}
		\hat\delta(k,\xi;-B)
		=1+\mathcal{O}(1),
		\quad k\to\infty;
	\end{equation}
	\item behavior at $k=-B$:
	\begin{equation*}
		\begin{split}
			&\hat\delta(k,\xi;-B)=
			(k+B)\frac{\delta_{1+}(-B,\xi;-B)
				\delta_2(-B,\xi;-B)}
			{\tilde{k}_0-B}
			+\mathcal{O}\left((k+B)^2\right),
			\quad k\to -B+\I0,\\
			&\hat\delta(k,\xi;-B)=
			\delta_{1-}(-B,\xi;-B)
			\delta_2(-B,\xi;-B)
			+\mathcal{O}(k+B),
			\quad k\to -B-\I0;
		\end{split}
	\end{equation*}
	\item singularity at $k=-\xi$ is described in \eqref{hat-del-s} with $-B$ instead of $B$.
\end{enumerate}
\end{description}
\begin{figure}
\begin{minipage}[h]{0.99\linewidth}
	\centering{\includegraphics[width=0.7\linewidth]{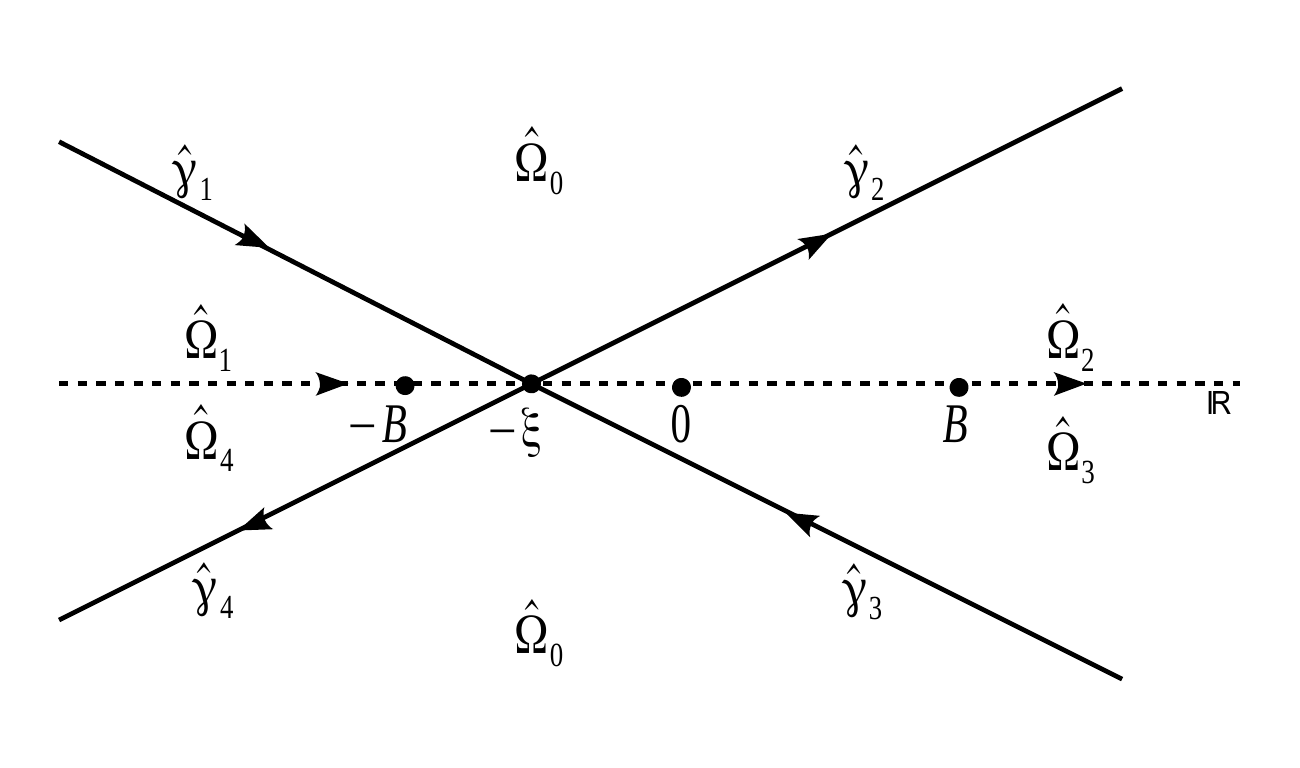}}
	\caption{Contour $\hat\Gamma
		=\hat\gamma_1\cup\dots
		\cup\hat\gamma_4$
		and domains $\hat\Omega_j$,
		$j=0,\dots,4$ in Section \ref{S2R2}.}
	%for $0\leq|\xi|<B$, $B>0$.}
\label{F-3}
\end{minipage}
\end{figure}
Making the same transformations $M\rightsquigarrow\tilde{M}
\rightsquigarrow\hat{M}$
as in
Section \ref{R1}, but with 
$\hat{\delta}(k,\xi;-B)$ instead of $\delta(k,\xi)$ and with
domains $\hat\Omega_j$, $j=0,\dots,4$,
and contour $\hat\Gamma$ given in Figure \ref{F-3}, we conclude that $\hat{M}(x,t,k)$ is bounded at both points $k=\pm B$ and it solves the following RH problem:
\begin{description}
\item [RH problem for $\hat{M}(x,t,k)$, $0\leq|\xi|<B$]
\indent
\begin{enumerate}
\item jump condition \eqref{hat-M-j}--\eqref{hat-J} with $\hat\delta(k,\xi;-B)$ (see \eqref{hat-del})
instead of $\delta(k,\xi)$ on the cross $\hat\Gamma$ given in Figure \ref{F-3};

\item normalization condition at infinity:
\begin{equation*}
	\hat{M}(x,t,k)=I
	+\mathcal{O}\left(k^{-1}\right)
	,\quad k\to\infty;
\end{equation*}

\item singularity at $k=-\xi$ (see \eqref{nu}):
\begin{equation*}
	\begin{split}
		&\hat{M}(x,t,k)
		\begin{pmatrix}
			(k+\xi)^{-\Imm\,\nu(-\xi)}& 0\\
			0& (k+\xi)^{\Imm\,\nu(-\xi)}
		\end{pmatrix}
		=\mathcal{O}(1),\quad
		k\to -\xi.
	\end{split}
\end{equation*}
\end{enumerate}
\end{description}

Thus, the solution $q(x,t)$ decays in the region $0\leq|\xi|<B$.

Now we are at the position to formulate the main result of Section \ref{S1}.
Applying the Deift-Zhou nonlinear steepest decent method \cite{DZ93} to the RH problem for $\hat{M}$, we
obtain the following long-time asymptotic behavior of $q(x,t)$
along the rays $\xi\in\R\setminus\{B,-B\}$:

\begin{theorem}[Case II, $n=0$]
\label{ThCII0}
Consider the Cauchy problem \eqref{bc2}-\eqref{ic} with $B\neq0$, where the
initial datum $q_0(x)$ satisfies the following properties:
\begin{enumerate}[i)]
\item $q_0(x)$ a compact perturbation of pure step function \eqref{shst};

\item spectral data $a_1(k)$, $a_2(k)$ and $b(k)$, associated to $q_0(x)$, satisfy Assumptions A;

\item $q_0(x)$ is such that $a_1(k)$ satisfies properties described in Case II with $n=0$.
\end{enumerate}
Then the main long-time asymptotic terms of the solution $q(x,t)$ along the rays $\xi=\frac{x}{4t}=const$ read as follows
(see Figures \ref{fas1} and \ref{fas2} for $n=0$):
\begin{enumerate}[1)]
\item if $\xi<-|B|$, then for any $B\neq0$ the asymptotics is described by the Zakharov-Manakov type formula:
\begin{equation*}
	%\label{as-sol-1}
	q(x,t)=t^{-\frac{1}{2}-\Imm\,\nu(\xi)}\alpha_1(\xi)
	\exp\left\{4\I t\xi^2-\I\Ree\,\nu(\xi)\ln t\right\}
	+R_1(\xi,t);
\end{equation*}
\item if $\xi>|B|$, then for any $B\neq0$ the asymptotics is described by the plane wave with modulated amplitude:
\begin{equation*}
	q(x,t)=A\delta^2(-B,\xi)e^{2\I Bx-4\I B^2t}
	+\mathcal{O}\left(t^{-\frac{1}{2}+|\Imm\,\nu(-\xi)|}\right);
\end{equation*}
\item if $0\leq|\xi|<|B|$, then
\begin{enumerate}[{3.}1)]
	\item for $B>0$ the asymptotics is described by the Zakharov-Manakov type formula:
	\begin{equation*}
		q(x,t)=t^{-\frac{1}{2}+\Imm\,\nu(-\xi)}
		\alpha_2(\xi)
		\exp\left\{4\I t\xi^2-\I\Ree\,\nu(-\xi)
		\log t\right\}+R_2(\xi,t);
	\end{equation*}
	\item for $B<0$ the asymptotics has the form of the periodic oscillations:
	\begin{equation*}
		q(x,t)=
		\frac{16AB^2\hat{\delta}^2(-B,\xi;B)
			e^{2\I Bx-4\I B^2t}}
		{16B^2-A^2\left(\overline{\hat{\delta}}\right)^2
			(-B,-\xi;B)
			\hat{\delta}^2(-B,\xi;B)e^{4\I Bx}}
		+\mathcal{O}\left(t^{-\frac{1}{2}+|\Imm\,\nu(-\xi)|}\right).
	\end{equation*}
\end{enumerate}
\end{enumerate}
Here $\nu(-\xi)$, $\delta(-B,\xi)$ and 
$\hat{\delta}(-B,\xi;B)$ are given in \eqref{nu} \eqref{ddef} and \eqref{hat-del}, respectively.
The asymptotic formula in item 3.2) holds uniformly in $x,t$ away from  arbitrarily small neighborhoods of the possible zeros of the denominator.
Notice that assumption iii) yields $-\frac{1}{2}<\Imm\,\nu(\xi)<\frac{1}{2}$, $\xi<-|B|$.
Constants $\alpha_j(\xi)$ and remainders $R_j(\xi,t)$, $j=1,2$, are as follows:
\begin{equation*}
\begin{split}
	&\alpha_1(\xi)=\dfrac{\sqrt{\pi}\,
		\exp\left\{-\frac{\pi}{2}\overline{\nu}(\xi)+\frac{\pi\I}{4}-2\overline{\chi}(\xi,-\xi)
		-3\I\overline{\nu}(\xi)\log2\right\}}
	{\overline{r_2(\xi)}\Gamma(-\I\overline{\nu}(\xi))},\\
	&\alpha_2(\xi)=\dfrac{\sqrt{\pi}
		\exp\left\{-\frac{\pi}{2}\nu(-\xi)
		+\frac{\pi\I}{4}+2(\chi_1+\chi_2)(-\xi,\xi;-B)-3\I\nu(-\xi)
		\log 2\right\}}
	{r_1(-\xi)\Gamma(-\I\nu(-\xi))},		
\end{split}
\end{equation*}
and
\begin{equation*}
%\label{R11}
R_1(\xi,t)=
\begin{cases}
	\mathcal{O}\left(t^{-1}\right),& \Imm\,\nu(\xi)>0,\\
	\mathcal{O}\left(t^{-1}\log t\right),&\Imm\,\nu(\xi)=0,\\
	\mathcal{O}\left(t^{-1+2|\Imm\,\nu(\xi)|}\right),&\Imm\,\nu(\xi)<0,
\end{cases}
\end{equation*}
\begin{equation*}
%\label{R21}
R_2(\xi,t)=
\begin{cases}
	\mathcal{O}\left(t^{-1+2|\Imm\,\nu(-\xi)|}\right),& 
	\Imm\,\nu(-\xi)>0,\\
	\mathcal{O}\left(t^{-1}\log t\right),&\Imm\,\nu(-\xi)=0,\\
	\mathcal{O}\left(t^{-1}\right),&\Imm\,\nu(-\xi)<0,
\end{cases}
\end{equation*}
with $\chi(\xi,-\xi)$ and $\chi_j(-\xi,\xi;-B)$, $j=1,2$ given in \eqref{chi} and \eqref{chi1,2} respectively.
\end{theorem}
\begin{proof}
Here the implementation of the Deift-Zhou nonlinear steepest decent method is based upon the approximation, as 
$t\to\infty$, of the RH problem for $\hat M$ 
%on the cross $\hat\Gamma$ 
in the vicinity of the stationary phase point $k=-\xi$, see Figures \ref{F-1}--\ref{F-3}.
To this end we introduce the so-called local parametrix, which is explicitly defined, as for the NLS equation, in terms of the parabolic cylinder functions \cite{I81}.
This parametrix provides good estimates for the RH problem for $\hat M$ as $t\to\infty$ and, consequently, for the solution $q(x,t)$ retrieved from $\hat M$ as follows (see \eqref{sol}--\eqref{sol1}):
\begin{equation*}
\begin{split}
	&q(x,t)=2\I\lim_{k\to\infty}
	k\hat{M}_{1,2}(x,t,k),\\
	&q(-x,t)=-2\I\lim_{k\to\infty}
	k\overline{\hat{M}_{2,1}}(x,t,k).
\end{split}
\end{equation*}

The details and peculiarities of the construction and asymptotic estimates of parametrix for the nonlocal NLS equation are presented in \cite{RS19, RS21-DE}.
Adhering to the analysis in \cite{RS19, RS21-DE}, we arrive at the long-time asymptotic expansion of $q(x,t)$ given in the Theorem.
\end{proof}

\subsection{Asymptotics in Case I with $n=0$}
\label{S2}
The difference between Case I, $n=0$, and Case II, $n=0$ consists in (i) the additional point of discrete spectrum at 
$k=\I k_0$, $k_0>0$ and (ii) the winding of the argument at $k=0$, see \eqref{aa1g}.
Since the corresponding residue condition, see \eqref{resinI}, vanishes exponentially fast as $|x|,t\to\infty$ along the rays 
$\xi\neq0$, we obtain the same problem for $\hat{M}$ in the regions $|\xi|>|B|$ and $0<|\xi|<|B|$ as in Sections \ref{R1}, \ref{S2R1} and \ref{R2}, \ref{S2R2}, respectively, for the large $x,t$.
Thus, the asymptotic behavior of $q(x,t)$ in Case I with $n=0$ is the same as in Theorem \ref{ThCII0} for all $\xi\neq0$, and we have the following theorem.
\begin{theorem}[Case I, $n=0$]
\label{ThCI0}
Consider the Cauchy problem \eqref{bc2}-\eqref{ic} with $B\neq0$, where the
initial datum $q_0(x)$ satisfies the following properties:
\begin{enumerate}[i)]
\item $q_0(x)$ a compact perturbation of pure step function \eqref{shst};

\item spectral data $a_1(k)$, $a_2(k)$ and $b(k)$, associated to $q_0(x)$, satisfy Assumptions A;

\item $q_0(x)$ is such that $a_1(k)$ satisfies properties described in Case I with $n=0$.
\end{enumerate}
Then the main long-time asymptotic terms of the solution $q(x,t)$ along the rays $\xi=\frac{x}{4t}=const$ in the regions 
$\xi<-|B|$, $\xi>|B|$ and $0<\xi<|B|$ are given in items 1), 2) and 3) of Theorem \ref{ThCII0}.
\end{theorem}

\section{Concluding remarks}
\label{CR}
\begin{enumerate}[i)]
\item Transition regions. 
In both Cases I and II with $n=0$, the resulting asymptotic picture lacks the description of the behavior of $q(x,t)$ along the rays $\xi=\pm B$.
From the perspective of the RH approach, these transition regions correspond to merging of the stationary phase point $k=-\xi$ and the singularities at $k=\pm B$, see \eqref{a1pmB}.
On the other hand, taking the limits  $\xi\uparrow\pm B$ and 
$\xi\downarrow\pm B$ in the main asymptotic terms in the corresponding regions, one concludes that the solution exhibits increasing oscillations.
A partial description of the transition region of this type can be found in \cite{RS21-SIAM} in the case $B=0$.

In Case I we have, in addition, a transition  at 
$\xi=0$. 
It is evident from Theorem \ref{ThCI0} that the asymptotics in the regions $-|B|<\xi<0$ and $0<\xi<|B|$ matches as $\xi\to0$. 
On the other hand, we cannot directly  analyze the RH problem for $\hat{M}$ 
asymptotically at $\xi=0$, since for such value of 
$\xi$ we have $\Imm\,\nu(0)=-\frac{1}{2}$ (see \eqref{aa1g} and \eqref{nu}), where the asymptotic analysis of the problem on the cross $\hat\Gamma$ fails, see \cite{RS19}.
Therefore the description of the transition zone(s) at $\xi=0$ remains an open question.

%$\xi=\pm B$ (
%correspond to singularities at $k=\pm B$, partial description in \cite{RS21-SIAM}), $\xi=\pm\Ree\,p_j$ (corr. to discrete spectrum, described by kink-type solitons \cite{RS21-CMP}), $\xi=\pm\omega_j$ (corr. to the winding of the argument, open problem)

\item Case I, II with arbitrary $n\in\N$.
In the region $|\xi|>|B|$, the major challenge is related to the winding of the argument of $1+r_1(k)r_2(k)$ at $k=\omega_j$, $j=1,\dots,n$, see items I.2), II.2) and \eqref{r-a}.
The ``intermittency'' property of $\omega_j$ and 
$\Ree\,p_j$, see \eqref{ordom}, is pivotal for resolving this problem, as discussed in \cite{RS21-CMP} (see also \cite{RS20-JMAG}).
Following, almost literally, the analysis of \cite{RS21-CMP}, it can be shown that the asymptotic regions $\xi<-|B|$ and $\xi>|B|$ split into $2n+1$ sectors having qualitatively different behavior.
Namely, for $\xi<-|B|$ we have $n+1$ decaying and $n$ modulated plane wave regions, while for $\xi>-|B|$ we have $n$ decaying and $n+1$ modulated plane wave regions, altering each other (see Figures \ref{fas1} and \ref{fas2} for the illustration).

Concerning the region $0<|\xi|<|B|$, the main difficulty consists in zero of $1+r_1(k)r_2(k)$ at $k=-|B|$, which is resolved in Section \ref{R2} (the singularity corresponding to the winding of the argument can be resolved in the same way as in \cite{RS21-CMP, RS20-JMAG}).

Since the asymptotic analysis for $n\in\N$ essentially repeats the methodology in Section \ref{S1} and \cite{RS21-CMP, RS20-JMAG}, we omit details here, and give the rough asymptotic formulas for the solution in the following proposition.

\begin{proposition}[Rough asymptotics in Cases I,II, $n\in\N$]
\label{pr1}
Consider the Cauchy problem \eqref{bc2}-\eqref{ic} with $B\neq0$, where the
initial datum $q_0(x)$ satisfies the following properties:
\begin{enumerate}[i)]
	\item $q_0(x)$ is a compact perturbation of pure step function \eqref{shst};
	
	\item spectral data $a_1(k)$, $a_2(k)$ and $b(k)$, associated to $q_0(x)$, satisfy Assumptions A;
	
	\item $q_0(x)$ is such that $a_1(k)$ satisfies properties described in either Case I or Case II with some $n\in\N$.
\end{enumerate}
Then the long-time asymptotic behavior of the solution $q(x,t)$ along the rays $\xi=\frac{x}{4t}=const$ has the following form (see \eqref{ordom}; to simplify notations we set $\Ree\,p_0:=-|B|$):
\begin{enumerate}[1)]
	
	\item if $-\Ree\,p_{n-m}<\xi<-\omega_{n-m+1}$, $m=0,\dots,n$ then 
	for any $B\neq0$
	\begin{equation*}
		%\label{as-sol-1}
		q(x,t)=A\xi^{2m}\delta^2(-B,\xi)
		\left(\prod\limits_{s=0}^{m-1}p_{n-s}^{-2}\right)
		e^{2\I Bx-4\I B^2t}
		+\mathcal{O}\left(t^{-\frac{1}{2}
			+|\Imm\,\nu(-\xi)-m|}\right);
	\end{equation*}
	
	\item if $\omega_{n-m+1}<\xi<\Ree\,p_{n-m}$, $m=0,\dots,n$, 
	then for any $B\neq0$
	\begin{equation*}
		q(x,t)=\mathcal{O}
		\left(t^{-\frac{1}{2}-\Imm\,\nu(\xi)+m}\right);
	\end{equation*}
	
	\item if $-\omega_{n-m}<\xi<-\Ree\,p_{n-m}$, $m=0,\dots,n-1$, 
	then for any $B\neq0$
	\begin{equation*}
		q(x,t)=\mathcal{O}
		\left(t^{-\frac{1}{2}+\Imm\,\nu(-\xi)-m}\right);
	\end{equation*}
	
	\item if $\Ree\,p_{n-m}<\xi<\omega_{n-m}$, $m=0,\dots,n-1$, then for any $B\neq0$
	\begin{equation*}
		%\label{as-sol-1}
		q(x,t)=\frac{-4\left(
			\prod\limits_{s=0}^{m}p_{n-s}^{2}\right)
			e^{-2\I Bx-4\I B^2t}}
		{A\xi^{2m}\overline{\delta}^2(-B,-\xi)}
		+\mathcal{O}\left(t^{-\frac{1}{2}
			+|\Imm\,\nu(-\xi)-m|}\right);
	\end{equation*}
	
	\item if $0<|\xi|<|B|$ in Case I or $0\leq|\xi|<|B|$ in Case II, then
	\begin{enumerate}[{5.}1)]
		\item for $B>0$
		\begin{equation*}
			q(x,t)=\mathcal{O}\left(t^{-\frac{1}{2}
				+|\Imm\,\nu(-\xi)-n|}\right);
		\end{equation*}
		\item for $B<0$
		\begin{equation*}
			\begin{split}
				q(x,t)=&
				\frac{16AB^2(B-\xi)^{2n}
					\prod\limits_{s=0}^{n-1}(B+p_{n-s})^{-2}\cdot
					\hat{\delta}^2(-B,\xi;B)
					e^{2\I Bx-4\I B^2t}}
				{16B^2-A^2
					\left(\frac{B-\xi}{B+\xi}\right)^{2n}
					\left(\prod\limits_{s=0}^{n-1}
					\frac{B-p_{n-s}}{B+p_{n-s}}
					\right)^2
					\left(\overline{\hat{\delta}}\right)^2
					(-B,-\xi;B)
					\hat{\delta}^2(-B,\xi;B)e^{4\I Bx}}\\
				&+\mathcal{O}\left(t^{-\frac{1}{2}+|\Imm\,\nu(-\xi)-n|}\right).
			\end{split}
		\end{equation*}
	\end{enumerate}
\end{enumerate}
Here $\nu(-\xi)$, $\delta(-B,\xi)$ and 
$\hat{\delta}(-B,\xi;B)$ are given in \eqref{nu} \eqref{ddef} and \eqref{hat-del}, respectively.
The asymptotic formula in item 5.2) holds uniformly in $x,t$ away from  arbitrarily small neighborhoods of the possible zeros of the denominator.
Notice that assumption iii) yields that in all considered asymptotic regions one has $-\frac{1}{2}<\Imm\,\nu(\xi)-m<\frac{1}{2}$ for the corresponding value of $m=0,\dots,n$.
\end{proposition}

\begin{figure}[h]
\begin{minipage}[h]{0.48\linewidth}
	\centering{\includegraphics[width=0.99\linewidth]{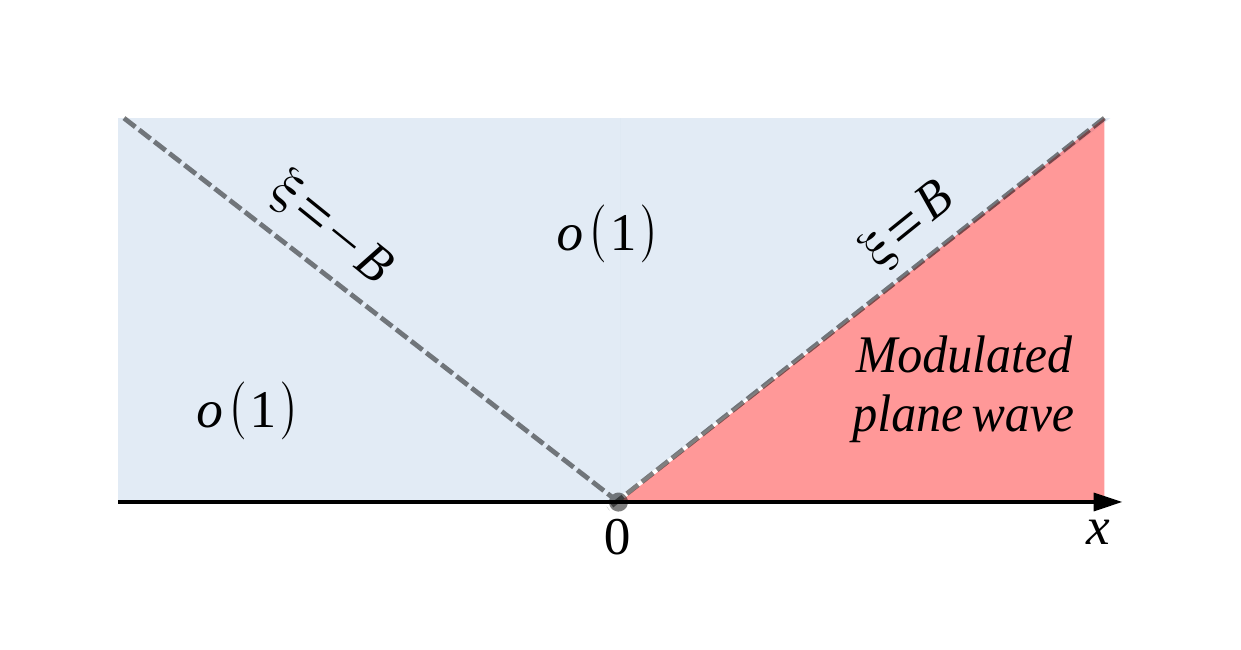}}
\end{minipage}
\hfill
\begin{minipage}[h]{0.48\linewidth}
	\centering{\includegraphics[width=0.99\linewidth]{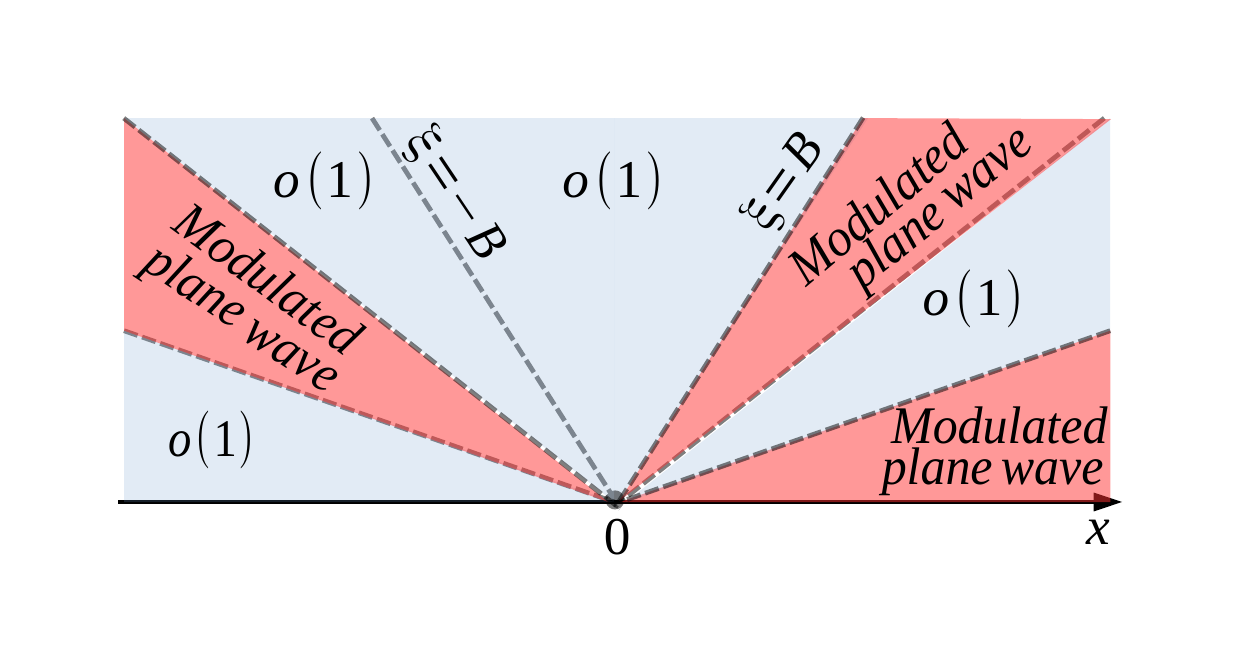}}
\end{minipage}
\caption{Asymptotics in Case II, $B>0$, for $n=0$ (left) and $n=1$ (right).}
\label{fas1}
\end{figure}

\begin{figure}[h]
\begin{minipage}[h]{0.48\linewidth}
	\centering{\includegraphics[width=0.99\linewidth]{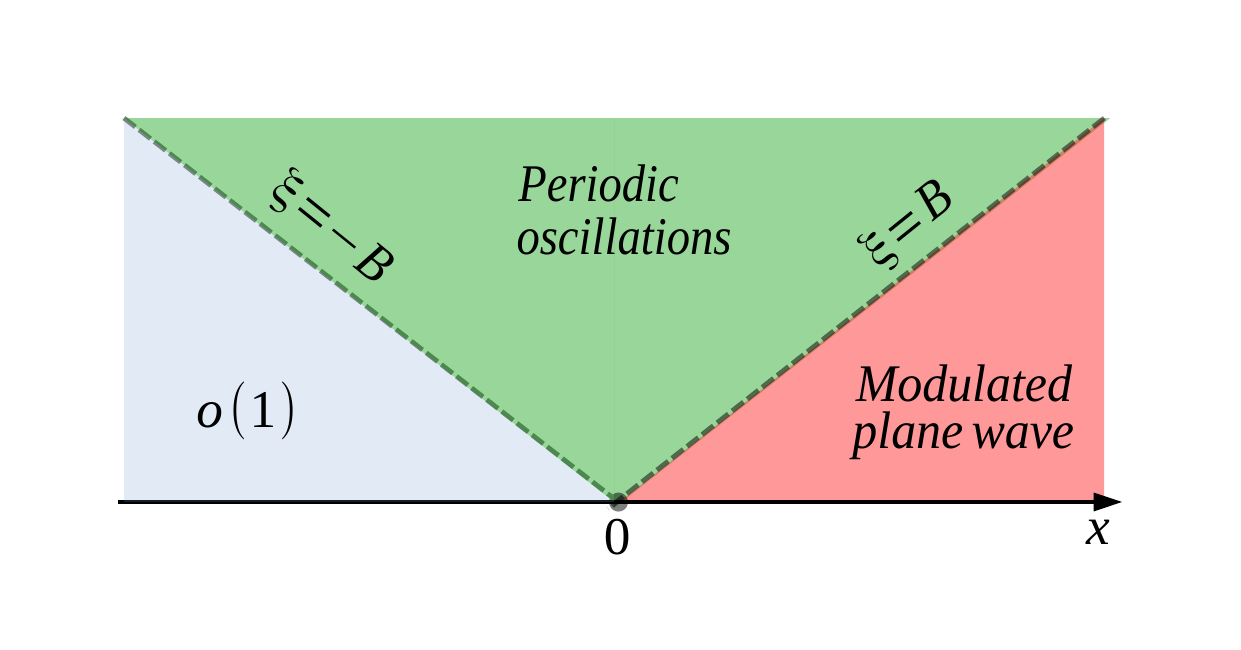}}
\end{minipage}
\hfill
\begin{minipage}[h]{0.48\linewidth}
	\centering{\includegraphics[width=0.99\linewidth]{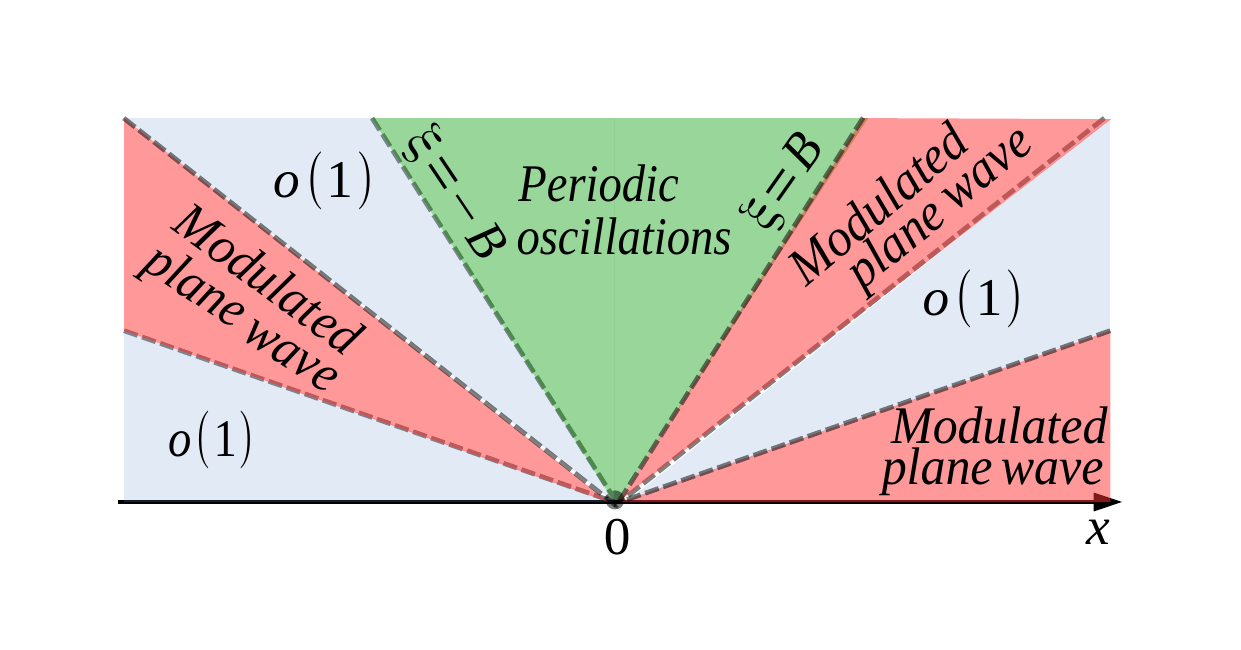}}
\end{minipage}
\caption{Asymptotics in Case II, $B<0$, for $n=0$ (left) and $n=1$ (right).}
\label{fas2}
\end{figure}

\item The Zakharov-Manakov type asymptotic formulas obtained in Theorems \ref{ThCII0} and \ref{ThCI0}  differ qualitatively from the classical one for the NLS equation \cite{ZM76}.
Namely, the power decay rate of the solutions depends, in general, on the direction $\xi$, along which we calculate the long-time asymptotic behavior of the solution.

\item The long-time asymptotic formula of $q(x,t)$ given in Theorem \ref{ThCII0}, items 1) and 3.2) and in Proposition \ref{pr1} can be extended by adding at least one precise decaying term, which is of the Zakharov-Manakov type, as detailed in 
\cite[Theorem 1]{RS21-DE, RS21-CMP}.

\end{enumerate}
\bigskip

\noindent {\bf Acknowledgements}. 
DS acknowledges the  support from the {QJMAM} Fund for Applied Mathematics.
The work of YR is supported by the European Union’s Horizon Europe research and innovation programme under the 
Marie Sk\l{}odowska-Curie grant agreement No 101058830.
S-FT acknowledges the support from the National Natural Science Foundation of China, Grant No. 12371255.
YR and DS thank the Armed Forces of Ukraine for providing security, which made this work possible. 

\appendix

\section{}\label{Ap1}
\begin{proof}[Proof of Proposition \ref{a1Zs}]
Introduce notation $k=k_1+\I k_2$, where $k_1\in\R$ and $k_2\geq0$.
Then equation $a_1(k)=0$ for 
$k\in\C^+\cup\R\setminus\{B,-B\}$
is equivalent to the following system of transcendental equations (see \eqref{spss}):
\begin{equation}\label{se}
\begin{split}
	&A^2e^{-4k_2R}\cos (4k_1R)=4B^2-4k_1^2+4k_2^2,\\
	&A^2e^{-4k_2R}\sin (4k_1R)=-8k_1k_2.
\end{split}
\end{equation}

To prove \textbf{item 1)}, we take
$k_2=0$ in \eqref{se} and obtain the equation
\begin{equation*}
A^2\cos (4k_1R)=4\left(B^2-k_1^2\right),\quad
\mbox{with }
4k_1R=\pi n,\,n\in\Z,
\end{equation*}
which directly yields items 1.1)--1.3).

%Observing that \eqref{tre} is nothing but \eqref{se} with $k_1=0$ and $k_2>0$, we obtain \textbf{item 2)}.
\textbf{Item 2)} follows from \eqref{se} with $k_1=0$ and $k_2>0$.

Let us prove \textbf{item 3)}.
For $R=0$, \eqref{spss} implies that $a_1(k)$ can have real or purely imaginary zeros only and thus we have 3.1) for $R=0$ and any $B\in\R$.
%In the case $B=0$ and $R>0$, item 3) was established in \cite[Proposition 2]{RS21-CMP}.

Now consider $R>0$.

\textbf{Step 1}.
Observing that $(k_1,k_2)$ and $(-k_1,k_2)$ satisfies \eqref{se} simultaneously (cf.\,\,symmetries in Proposition \ref{prsp}, item (3)), it is enough to study \eqref{se} for $k_1,k_2>0$ only.
Introduce notations
\begin{equation*}
\tau=4k_1R,\quad y=4k_2R,\quad \tau,y>0.
\end{equation*}
Then \eqref{se} reads
\begin{equation}\label{sety}
\begin{split}
	&4A^2R^2e^{-y}\cos\tau=16B^2R^2-\tau^2+y^2,\\
	&4A^2R^2e^{-y}\sin\tau=-2\tau y.
\end{split}
\end{equation}
Since $\tau,y>0$, the second equation in \eqref{sety} implies that
\begin{equation}\label{rest}
\tau\in(2\pi n-\pi,2\pi n),\quad n\in\N.
\end{equation}
Dividing the first equation in \eqref{sety} by the second one, we arrive at the following quadratic equation for $y$:
\begin{equation*}
y^2+2\tau y\cot\tau-\tau^2+16B^2R^2=0,
\end{equation*}
which has a single positive solution
\begin{equation}\label{yexp}
y=-\tau\cot\tau+\left(
\tau^2\cot^2\tau+\tau^2-16B^2R^2
\right)^{1/2}.
\end{equation}
Combining \eqref{sety}, \eqref{rest} and \eqref{yexp}, we obtain the following system:
\begin{equation}\label{sety1}
\begin{split}
	&y=-\tau\cot\tau+\left(
	\tau^2\cot^2\tau+\tau^2-16B^2R^2
	\right)^{1/2},\quad
	\tau\in(2\pi n-\pi,2\pi n),\, n\in\N,\\
	&ye^y=-2A^2R^2\frac{\sin\tau}{\tau}.
\end{split}
\end{equation}
Notice that since we assume that $4R|B|\leq \pi$ and $\tau>\pi$, the function under the square root in \eqref{sety1} is always positive.

\textbf{Step 2}. Let us investigate the behavior of $y=y(\tau)$ for
$\tau\in(2\pi n-\pi,2\pi n)$, $n\in\N$.
It is evident that
\begin{equation*}
\lim\limits_{\tau\uparrow2\pi n}y(\tau)=\infty,
\quad n\in\N.
\end{equation*}
Taking into account that
\begin{equation}\label{y0}
\left(
1+\frac{\tau^2-16B^2R^2}{\tau^2\cot^2\tau}    
\right)^{1/2}
=1+\frac{\tau^2-16B^2R^2}{2\tau^2\cot^2\tau}
+\mathcal{O}\left(\cot^{-4}\tau\right),
\quad \tau\downarrow -\pi+2\pi n,\,n\in\N,
\end{equation}
we obtain the following limit:
\begin{equation*}
\lim\limits_{\tau\downarrow-\pi+2\pi n}y(\tau)=0.
\end{equation*}

\textbf{Step 3}. Now we study $y^\prime(\tau)$, $\tau\in(2\pi n-\pi,2\pi n)$, $n\in\N$.
Direct calculations show that
\begin{equation}\label{dy}
\begin{split}
	&y^\prime(\tau)=
	\frac{\left(\tau-\frac{1}{2}\sin(2\tau)\right)
		\left(\tau\cos\tau
		+d(\tau)\right)
		-\tau\sin^3\tau}
	{d(\tau)\sin^2\tau},\\
	&d(\tau)=\left(\tau^2-16B^2R^2\sin^2\tau\right)^{1/2}.
\end{split}
\end{equation}
It is easy to see that
\begin{equation*}
\lim\limits_{\tau\uparrow2\pi n}
y^\prime(\tau)=\infty,
\quad n\in\N.
\end{equation*}
Using the following relations:
\begin{equation*}
\begin{split}
	&\tau-\frac{1}{2}\sin2\tau=2\pi n-\pi
	+\mathcal{O}\left(
	(\tau+\pi-2\pi n)^3
	\right),\\
	&\cos\tau
	+\frac{d(\tau)}{\tau}=
	\cos\tau
	+\left(1-16B^2R^2\frac{\sin^2\tau}
	{\tau^2}\right)^{1/2}\\
	&\qquad\qquad\quad\,\,\,\,
	=\left(\frac{1}{2}-\frac{8B^2R^2}{\pi^2(2n-1)^2}\right)
	(\tau+\pi-2\pi n)^2
	+\mathcal{O}\left(
	(\tau+\pi-2\pi n)^4
	\right),
\end{split}
\end{equation*}
where $\tau\to2\pi n-\pi$, $n\in\N$, we obtain
\begin{equation}\label{yp0}
\lim\limits_{\tau\downarrow2\pi n-\pi}
y^\prime(\tau)=
\frac{\pi^2(2n-1)^2-16B^2R^2}{2\pi(2n-1)},
%\pi(2n-1)
%\left(\frac{1}{2}-\frac{8B^2R^2}{\pi^2(2n-1)^2}\right),
\quad n\in\N.
\end{equation}

\textbf{Step 4}. 
Let us prove that $y(\tau)$ is convex for 
$\tau\in(2\pi n-\pi,2\pi n)$, $n\in\N$.
Direct computations show that
\begin{equation}\label{y2p}
\begin{split}
	&y^{\prime\prime}(\tau)=\frac{-g(\tau)}{d^3(\tau)\sin^3\tau},\\
	&g(\tau)=d^2(\tau)\left(
	\tau^2(1+\cos^2\tau)+d(\tau)(2\tau\cos\tau-\sin\tau)
	-\tau\sin(2\tau)+\sin^2\tau
	\right)\\
	&\qquad\quad-\tau\sin\tau
	\left(\tau-8B^2R^2\sin 2\tau\right)
	(\sin\tau-\tau\cos\tau),
\end{split}
\end{equation}
where $d(\tau)$ is given in \eqref{dy}.
We must prove that $g(\tau)\geq0$ for $\tau\in(2\pi n-\pi,2\pi n)$, $n\in\N$.
It is easy to compute that
\begin{equation*}
g(2\pi n-\pi)=0,\quad
g(2\pi n)=2^6(\pi n)^4,\quad n\in\N.
\end{equation*}
The most involving task is to show that $g(\tau)\geq$ in the vicinity of $\tau=2\pi n-\pi$.
Using the following expansions:
\begin{equation*}
\begin{split}
	&\cos\tau=-1+\frac{(\tau+\pi-2\pi n)^2}{2}
	+\mathcal{O}\left(
	(\tau+\pi-2\pi n)^4
	\right),\\
	&\sin\tau=-(\tau+\pi-2\pi n)
	+\frac{(\tau+\pi-2\pi n)^3}{6}
	+\mathcal{O}\left(
	(\tau+\pi-2\pi n)^4
	\right),\\
	&d(\tau)=\tau-\frac{8B^2R^2}{\tau}
	(\tau+\pi-2\pi n)^2
	+\mathcal{O}\left(
	(\tau+\pi-2\pi n)^4
	\right),
\end{split}
\end{equation*}
where $\tau\to2\pi n-\pi$, $n\in\N$, we conclude that
\begin{equation*}
g(\tau)=\frac{\tau}{6}
\left(144B^2R^2+17\tau^2\right)
(\tau+\pi-2\pi n)^3
+\mathcal{O}\left(
(\tau+\pi-2\pi n)^4
\right),\quad
\tau\to\pi-2\pi n,\,n\in\N
\end{equation*}
and thus   $g(\tau)>0$ for 
$\tau$ near $2\pi n-\pi$.

\textbf{Step 5.}
Now we are at the position to study the solutions of \eqref{sety1}.
Using \eqref{y0} and \eqref{yp0}, we obtain
\begin{equation}\label{dl}
\lim\limits_{\tau\downarrow2\pi n-\pi}
\left(y(\tau)e^{y(\tau)}\right)^\prime
=\frac{\pi^2(2n-1)^2-16B^2R^2}{2\pi(2n-1)},
\quad n\in\N.
\end{equation}
Also we have the following identity:
\begin{equation}\label{dr}
\lim\limits_{\tau\downarrow2\pi n-\pi}
\left(\frac{\sin\tau}{\tau}\right)^\prime
=\frac{-1}{\pi(2n-1)},
\quad n\in\N.
\end{equation}
Then we notice that
\begin{enumerate}[(i)]
\item $y(\tau)e^{y(\tau)}$ is convex for 
$\tau\in (2\pi n-\pi,2\pi n)$, see Step 4 above;
\item $\left(-\frac{\sin\tau}{\tau}\right)$ 
has a unique maximum point 
$\tau_n^*$ on the interval
$(2\pi n-\pi,2\pi n)$, which is a unique solution of the equation $\tau_n^*=\tan\tau_n^*$, 
$\tau_n^*\in (2\pi n-\pi,2\pi n-\pi/2)$;
\item $\left(-\frac{\sin\tau}{\tau}\right)$
is concave for 
$\tau\in(2\pi n-\pi,\tau_n^*)$.
\end{enumerate}
Leveraging (i)--(iii) above and using
\eqref{dl} and \eqref{dr}, we conclude that\\
if $4A^2R^2\leq \pi^2(2n-1)^2-16B^2R^2$ for some $n\in\N$, then \eqref{sety1} has no solutions for such a $n$, while\\
if $4A^2R^2>\pi^2(2n-1)^2-16B^2R^2$ for some $n\in\N$, then \eqref{sety1} has a unique simple solution for such a $n$.
%see Figure ... for the illustration.
This yields item 3).
\end{proof}

\section{}\label{Ap2}
\begin{proof}[Proof of Proposition \ref{a1Ws}]
Notice that
\begin{equation}\label{ria1}
\Ree\,a_1(k)=
1+\frac{A^2\cos 4kR}{4(k^2-B^2)},\quad
\Imm\,a_1(k)=
\frac{A^2\sin 4kR}{4(k^2-B^2)}.
\end{equation}
\textbf{Step 1.}
Consider the case
$0<R<\frac{\pi}{2\left(4B^2+A^2\right)^{1/2}}$.
Then we have from \eqref{ria1} that
$\Ree\,a_1(-\pi/4R)>0$,
which implies that
$\Ree\,a_1(k)>0$ for all
$k\in(-\infty,-\pi/4R]$ and therefore
$\int_{-\infty}^k d\arg(a_1(k))\in[-\pi/2,\pi/2]$,
$k\in(-\infty,-\pi/4R]$.
Taking into account that 
$\sin 4kR<0$ for all $k\in(-\pi/4R,-|B|)$, we conclude that
$\int_{-\infty}^k d\arg(a_1(k))\in(-\pi,0)$ for all
$k\in(-\pi/4R,-|B|)$ and the first statement in \eqref{R0} is established.

Observing that 
\begin{equation*}
\begin{split}
	&\lim\limits_{k\uparrow-|B|}
	\tan\left(\arg a_1(k)\right)
	=-\tan(4|B|R),\quad
	4|B|R\neq\frac{\pi}{2},\\
	&\lim\limits_{k\uparrow-|B|}
	\tan\left(\arg a_1(k)\right)
	=-\infty,\quad
	4|B|R=\frac{\pi}{2},
\end{split}
\end{equation*}
we arrive at the second equation in \eqref{R0}.
Then using that $\frac{\sin 4kR}{k^2-B^2}>0$ for
$k\in(-|B|,0)$, we obtain \eqref{R0-1}.
Finally, since 
$a_1(0)>0$ and $a_1(0)<0$ for
$4B^2-A^2>0$ and $4B^2-A^2<0$ respectively,
%in the case $4B^2-A^2>0$ we have $a_1(0)>0$, while in the case $4B^2-A^2<0$ we have $a_1(0)<0$, 
we conclude
\eqref{arga10} for $0<R<\frac{\pi}{2\left(4B^2+A^2\right)^{1/2}}$.

\textbf{Step 2}.
Consider
$\frac{(2n-1)\pi}
{2\left(4B^2+A^2\right)^{1/2}}<R<\frac{(2n+1)\pi}
{2\left(4B^2+A^2\right)^{1/2}}$ for some $n\in\N$.
Taking into account that (see \eqref{om1})
\begin{enumerate}[(i)]
\item $\Ree\,a_1(k)>0$ for all
$k\leq\omega_n-\pi/2R$ and $k=\omega_n-\pi/4R$,
\item $\sin 4kR<0$ for 
$k\in(\omega_n-\pi/2R,\omega_n-\pi/4R)$ and
$\sin 4kR>0$ for
$k\in(\omega_n-\pi/4R,\omega_n)$,
%\item $\Ree\,a_1(k)<0$ for all $k=\omega_j$, 
%$j=1,\dots,n$;
\end{enumerate}
we conclude that
$\int_{-\infty}^kd\arg a_1(k)\in(-\pi,\pi)$ for $k<\omega_n$ and
$\int_{-\infty}^{\omega_n}d\arg a_1(k)=\pi$, see \eqref{ria1}. 
Thus, the first and the second equations in \eqref{Rn} are proved for $j=1$ and $j=0$ respectively.

Then using that
\begin{enumerate}[(i)]
\item $\Ree\,a_1(k)<0$ for all $k=\omega_j$, 
$j=1,\dots,n$,
\item $\sin 4kR<0$ for 
$k\in(\omega_{j+1},
\omega_{j+1}+\pi/4R)$ and
$\sin 4kR>0$ for
$k\in(\omega_{j+1}+\pi/4R,\omega_j)$,
$j=1,\dots, n-1$,
\item $\sin 4kR<0$ for 
$k\in(\omega_1,-|B|)$,
\end{enumerate}
we arrive at \eqref{Rn} for all $j$.
The proofs of \eqref{Rn-1} and item 3) are the same as in Step 1 above.
\end{proof}
	
\end{document}